\documentclass[11pt,a4paper]{article}
\usepackage[T1]{fontenc}
\usepackage[utf8]{inputenc}
\usepackage{lmodern}
\usepackage[margin=27mm]{geometry}
\usepackage{microtype}

\usepackage{amsmath,amssymb,amsfonts,amsthm}
\usepackage{mathrsfs}
\usepackage{bm}
\usepackage{enumitem}
\usepackage[hidelinks]{hyperref}
\usepackage[nameinlink]{cleveref}

\newtheorem{theorem}{Theorem}[section]
\newtheorem{lemma}[theorem]{Lemma}
\newtheorem{proposition}[theorem]{Proposition}

\theoremstyle{remark}
\newtheorem{remark}[theorem]{Remark}
\theoremstyle{plain}

\numberwithin{equation}{section}

\newcommand{\diver}{\operatorname{div}}

\hypersetup{
  pdftitle={Uniform Regularity for the Compressible Navier-Stokes-Poisson System under Navier-Slip Boundary Conditions},
  pdfauthor={Hairong Liu, Jinlong Wu},
  pdfsubject={Partial differential equations},
  pdfkeywords={Navier-Stokes-Poisson equations, Navier-slip boundary condition, conormal Sobolev spaces, uniform regularity}
}
\allowdisplaybreaks[2]
\title{\Large Uniform Regularity for the Compressible\\
Navier--Stokes--Poisson System\\
under Navier--Slip Boundary Conditions}
\author{Hairong Liu\quad Jinlong Wu\thanks{Corresponding author.
Email addresses: \href{mailto:hrliu@njust.edu.cn}{\texttt{hrliu@njust.edu.cn}}
(Hairong Liu), \href{mailto:jinlongno@njust.edu.cn}{\texttt{jinlongno@njust.edu.cn}}
(Jinlong Wu).}\\[0.7em]
\small School of Mathematics and Statistics,\\
\small Nanjing University of Science and Technology,\\
\small Nanjing 210094, P.R. China}
\date{}

\begin{document}
\maketitle

\begin{abstract}
In this paper, we establish the existence of a local-in-time strong solution to the compressible Navier--Stokes--Poisson system in a three-dimensional bounded smooth domain, subject to the general Navier--slip boundary condition. The solution is shown to be uniformly bounded in conormal Sobolev spaces, with uniform regularity estimates that are independent of the viscosity coefficients. As a consequence, we obtain the inviscid limit to the corresponding compressible Euler--Poisson system via a strong compactness argument.
\end{abstract}

\noindent\textbf{Keywords:} Compressible Navier--Stokes--Poisson equations;
Navier--slip boundary condition; conormal Sobolev spaces; uniform regularity.

\section{Introduction and main results}\label{sec:intro}

The compressible Navier--Stokes--Poisson (NSP) system is a classical model that couples the Navier--Stokes equations with a self-consistent Poisson equation, and is commonly employed to describe the motion of charged particles such as electrons or holes.
In this paper, we study the initial-boundary value problem for the 
compressible Navier--Stokes--Poisson system in a three-dimensional smooth
bounded domain \(\Omega\subset\mathbb R^3\):
\begin{equation}\label{eq:intro_NSP}
\left\{
\begin{aligned}
&\rho_t^\varepsilon+\diver(\rho^\varepsilon u^\varepsilon)=0,\\
&\rho^\varepsilon u_t^\varepsilon
+\rho^\varepsilon u^\varepsilon\cdot\nabla u^\varepsilon
+\nabla p^\varepsilon
=
\mu\varepsilon\Delta u^\varepsilon
+(\lambda+\mu)\varepsilon\nabla\diver u^\varepsilon
+\rho^\varepsilon\nabla\phi^\varepsilon,\\
&\Delta\phi^\varepsilon=\rho^\varepsilon-1,
\end{aligned}
\right.
\end{equation}
where \(\rho^\varepsilon\), \(u^\varepsilon\), and \(\phi^\varepsilon\)
denote the density, velocity, and electric potential, respectively. The
pressure is given by
\[
p^\varepsilon=(\rho^\varepsilon)^\gamma,\qquad \gamma>1,
\]
and the viscosity coefficients satisfy
\[
\mu>0,\qquad 2\mu+3\lambda>0.
\]
Here \(0<\varepsilon\leq1\) is the viscosity parameter, and the background
charge density in the Poisson equation is normalized to be \(1\).

We impose a general Navier--slip boundary condition on the velocity. In its
stress form, this condition is written as
\[
u^\varepsilon\cdot n=0,\qquad
(Su^\varepsilon\cdot n)_\tau=-(A u^\varepsilon)_\tau
\quad\text{on }\partial\Omega,
\]
where
\[
S u^\varepsilon
=
\frac12\left(
\nabla u^\varepsilon+(\nabla u^\varepsilon)^T
\right),
\]
\(n\) is the unit outward normal vector to \(\partial\Omega\), \(A\) is a given
smooth symmetric matrix on the boundary, and \((\cdot)_\tau\) denotes the
tangential component. For smooth solutions, using the identity
\[
\left(2S(v)n-(\nabla\times v)\times n\right)_\tau
=
-\left(2S(n)v\right)_\tau ,
\]
this boundary condition can be equivalently rewritten in the vorticity form
\begin{equation}\label{eq:intro_Navier_boundary}
u^\varepsilon\cdot n=0,\qquad
n\times\omega^\varepsilon=[B u^\varepsilon]_\tau
\quad\text{on }\partial\Omega,
\end{equation}
where
\[
\omega^\varepsilon=\nabla\times u^\varepsilon,
\qquad
B=2(A-S(n)).
\]
This vorticity formulation is more convenient for the normal derivative and
boundary vorticity estimates. For the electric potential, we impose the
homogeneous Neumann boundary condition
\begin{equation}\label{eq:intro_phi_boundary}
\partial_n\phi^\varepsilon=0
\quad\text{on }\partial\Omega.
\end{equation}

As \(\varepsilon\to0\), the viscous terms in
\eqref{eq:intro_NSP} formally vanish, and the corresponding limiting system is
the compressible Euler--Poisson system
\begin{equation}\label{eq:intro_EP}
\left\{
\begin{aligned}
&\rho_t+\diver(\rho u)=0,\\
&\rho u_t+\rho u\cdot\nabla u+\nabla p
=
\rho\nabla\phi,\\
&\Delta\phi=\rho-1,
\end{aligned}
\right.
\qquad x\in\Omega,\quad t>0,
\end{equation}
where \(p=\rho^\gamma\). The limiting boundary conditions are
\begin{equation}\label{eq:intro_EP_boundary}
u\cdot n=0,\qquad \partial_n\phi=0
\quad\text{on }\partial\Omega .
\end{equation}
Thus, our aim is to establish uniform estimates for
\eqref{eq:intro_NSP}--\eqref{eq:intro_phi_boundary} in a conormal Sobolev
framework and to justify the inviscid limit from
\eqref{eq:intro_NSP} to \eqref{eq:intro_EP}.

The vanishing viscosity limit for the Navier--Stokes equations is a classical
problem in mathematical fluid mechanics. In the absence of boundaries, the
limit can be justified in the standard Sobolev framework as long as the
corresponding Euler solution remains smooth \cite{Kato1972}. In domains with
physical boundaries, however, the problem becomes much more delicate because
viscous and inviscid flows usually satisfy different boundary conditions. For
the no-slip boundary condition, the inviscid limit is closely related to the
Prandtl boundary layer theory and remains largely open, except in some special
settings such as analytic data, vorticity initially separated from the
boundary, or certain steady regimes
\cite{SammartinoCaflisch1998I,SammartinoCaflisch1998II,Maekawa2014,
NguyenNguyen2018}. In contrast, Navier--slip type boundary conditions are more
compatible with the inviscid boundary condition and therefore make it possible
to establish uniform regularity estimates. For incompressible flows, uniform
regularity and the inviscid limit under slip or Navier-type boundary
conditions were studied in
\cite{XiaoXin2007,BeiraoCrispoCrispo2011,IftimieSueur2011,XiaoXin2013}.
A decisive progress was made by Masmoudi and Rousset
\cite{MasmoudiRousset2012}, who developed conormal Sobolev estimates in
general domains and justified the vanishing viscosity limit by compactness.
This conormal approach was later adapted to the compressible Navier--Stokes
equations with general Navier--slip boundary conditions by Wang, Xin and Yong
\cite{WangXinYong2015}, and to the full compressible Navier--Stokes system by
Wang \cite{Wang2016}. These works provide the main analytical background for
the present paper, where conormal derivatives are used to capture tangential
regularity while the missing normal regularity is recovered from the equations
and the boundary structure.

The main object of this paper is the compressible Navier--Stokes--Poisson
system. This model describes the motion of a compressible charged fluid coupled
with a self-consistent electrostatic field and appears in plasma physics,
semiconductor theory, ionized gas dynamics and related hydrodynamic models of
charged particles. A systematic introduction to the compressible
Navier--Stokes--Poisson equations can be found in Hsiao and Li
\cite{HsiaoLi2010}. Compared with the usual compressible Navier--Stokes
equations, the Poisson equation introduces a nonlocal elliptic coupling into
the momentum equation, and the electric force has to be estimated together
with the fluid variables. This changes both the energy structure and the
regularity analysis, especially in the presence of physical boundaries.

There is a large literature on the well-posedness and large-time behavior of
the Navier--Stokes--Poisson system. Weak solutions were studied, for example,
in \cite{Donatelli2003,KobayashiSuzuki2008,ZhangTan2007}, where the main
difficulty is to combine the compactness of the fluid variables with the
elliptic structure of the self-consistent potential. For strong solutions in
the Cauchy problem, global existence and decay around equilibrium states have
been obtained in various perturbative frameworks
\cite{HaoLi2009,LiMatsumuraZhang2010,TanWu2012,Wang2012,
TanYangZhaoZou2013,BieWangYao2017}. These works show that the electric field
may produce additional dissipation or dispersive effects, but at the same time
creates new coupling terms in the high-order energy estimates. For problems
with boundaries or nontrivial background profiles, stability of steady states,
global well-posedness and decay have also been investigated in bounded or
exterior domains
\cite{TanWangWang2015,LiuLuoZhong2020,LiuZhong2021,ChenHuangShi2024}.
In particular, the treatment of boundary conditions for the electric potential
is essential, since different choices of boundary conditions may lead to
different boundary mechanisms. There are also related models with additional
physical effects, such as damping term, heat conductivity, capillarity or
quantum/Korteweg-type corrections
\cite{WangChen1998,Wang2001,Jungel2010,AntonelliSpirito2022}. These variants
further illustrate that the Poisson coupling interacts strongly with the
dissipative and dispersive structures of the fluid equations.

Singular limits for the Navier--Stokes--Poisson system form another important
direction. The quasi-neutral limit, the vanishing Debye length limit, the
relaxation limit and the vanishing viscosity limit have been studied by
compactness methods, relative entropy methods, dispersive estimates, boundary
layer analysis and energy estimates. In the whole space or periodic domains,
the convergence from Navier--Stokes--Poisson type systems to incompressible
Euler, Euler--Poisson or related limiting systems has been considered in
\cite{WangJiang2006,DonatelliMarcati2008,JuLiLi2009,LiJuXu2015}.
For half-space or boundary problems, the interaction between the Poisson
equation and the boundary conditions becomes more subtle. Ju and Xu
\cite{JuXu2018} studied quasi-neutral and zero-viscosity limits for the
Navier--Stokes--Poisson equations in the half-space. Ju, Luo and Xu
\cite{JuLuoXu2023} justified singular limits for a viscous plasma under
Navier--slip boundary condition for the velocity and Dirichlet boundary
condition for the electric potential; in that problem, the main difficulty is
the strong density boundary layer generated by the prescribed boundary
potential and its interaction with the weak velocity boundary layer. In
contrast, when the electric potential satisfies a homogeneous Neumann boundary
condition, the electrostatic energy is naturally compatible with integration
by parts, and the strong boundary layer mechanism caused by the prescribed
potential is absent. Nevertheless, uniform high-order estimates remain
nontrivial because one still has to control the nonlocal electric field
together with the compressible fluid variables.

Recently, uniform regularity and vanishing viscosity limits for
Navier--Stokes--Poisson type systems have attracted further attention. For the
compressible Navier--Stokes--Poisson system with damping term under the slip
boundary condition, Gao, Peng, Wei and Yao \cite{GaoPengWeiYao2026}
established global-in-time uniform bounds and obtained the vanishing viscosity
limit with convergence rates. The damping term plays a stabilizing role in
their analysis and allows global-in-time estimates. The present paper is
concerned with a different setting: we consider the compressible
Navier--Stokes--Poisson system without damping term in a general
three-dimensional bounded domain, equipped with the Navier--slip boundary
condition for the velocity and the homogeneous Neumann boundary condition for
the electrostatic potential. Therefore, the problem is local in time, but it
retains the essential difficulties caused by the curved boundary, the
compressibility and the nonlocal Poisson coupling.

The goal of this paper is to establish uniform conormal regularity estimates
for the Navier--Stokes--Poisson system, independently of the viscosity, and
then to justify the vanishing viscosity limit to the corresponding
Euler--Poisson system. Compared with the compressible Navier--Stokes theory
\cite{WangXinYong2015,Wang2016}, the new difficulty is the electric force
\(\rho\nabla\phi\), which requires elliptic estimates for the Poisson equation
at the same conormal level as the fluid estimates. In particular, one has to
control \(\nabla\phi\) in both conormal Sobolev norms and \(L^\infty\)-type
norms, and to show that the Poisson coupling can be incorporated into the
energy method without losing derivatives. The homogeneous Neumann boundary
condition for \(\phi\) is crucial in this respect. Our analysis shows that the
conormal framework remains effective for the Navier--Stokes--Poisson coupling
in general bounded domains and yields, by compactness, the inviscid
Euler--Poisson limit.

We now state the main results. The system and boundary
conditions have been given above. Throughout this paper, for each
\(0<\varepsilon\leq1\), the initial potential \(\phi_0\) is
determined by
\[
\Delta\phi_0=\rho_0-1,
\qquad
\partial_n\phi_0=0 \quad\text{on } \partial\Omega .
\]
We also assume the compatibility condition
\[
\int_\Omega(\rho_0-1)\,dx=0.
\]

We define the functional space adapted to the viscous
Navier--Stokes--Poisson system. For \(0<\varepsilon\leq1\), set
\begin{equation}\label{eq:Xepsilon_m_norm}
\begin{aligned}
\|(p,u,\nabla\phi)(t)\|_{X_m^\varepsilon}^2
:={}&
\|(p,u,\nabla\phi)(t)\|_{\mathcal H^m}^2
+
\|\nabla u(t)\|_{\mathcal H^{m-1}}^2 \\
&
+
\sum_{k=0}^{m-2}
\|\partial_t^k\nabla p(t)\|_{m-1-k}^2  
+
\|\Delta p(t)\|_{\mathcal H^1}^2 \\
&
+
\|(\nabla u,\nabla\phi)(t)\|_{\mathcal H^{1,\infty}}^2  
+
\varepsilon\|\nabla\partial_t^{m-1}p(t)\|_{L^2}^2
+
\varepsilon\|\Delta p(t)\|_{\mathcal H^2}^2 .
\end{aligned}
\end{equation}
For the limiting Euler--Poisson system \eqref{eq:intro_EP}, we use the notation
\begin{equation}\label{eq:Xzero_m_norm}
\begin{aligned}
\|(p,u,\nabla\phi)(t)\|_{X_m^0}^2
:={}&
\|(p,u,\nabla\phi)(t)\|_{\mathcal H^m}^2
+
\|\nabla u(t)\|_{\mathcal H^{m-1}}^2\\
&
+
\sum_{k=0}^{m-2}
\|\partial_t^k\nabla p(t)\|_{m-1-k}^2  
+
\|\Delta p(t)\|_{\mathcal H^1}^2\\
&
+
\|(\nabla u,\nabla\phi)(t)\|_{\mathcal H^{1,\infty}}^2 .
\end{aligned}
\end{equation}

Here \(\|\cdot\|_{\mathcal H^m}\) denotes the space-time conormal Sobolev norm,
while \(\|\cdot\|_k\) denotes the pure spatial conormal Sobolev norm.
Moreover, \(\|\cdot\|_{\mathcal H^{1,\infty}}\) denotes the corresponding
conormal \(L^\infty\)-type norm. Their precise definitions are given in
Section~\ref{sec:preliminaries}.

Compared with the space used for the compressible Navier--Stokes equations, the
space in \eqref{eq:Xepsilon_m_norm} contains the additional electric-field term
\(\|\nabla\phi(t)\|_{\mathcal H^m}^2\), which is required to control the force
\(\rho\nabla\phi\) coupled through the Poisson equation.

To complete the setup of our functional framework, we define the uniform solution space $X_m^\varepsilon(T)$ as
\begin{equation}\label{eq:Xepsilon_m_T_space}
X_m^\varepsilon(T)
=
\left\{
(p,u,\nabla\phi)\in L^\infty([0,T];L^2(\Omega))
\ \Big|\
\operatorname*{ess\,sup}_{0\leq t\leq T}
\|(p,u,\nabla\phi)(t)\|_{X_m^\varepsilon}<+\infty
\right\}.
\end{equation}

We assume that the initial data satisfy
\begin{equation}\label{eq:initial_density_bound_main}
0<\frac1{\widehat C_0}
\leq
\rho_0(x)
\leq
\widehat C_0
<\infty,
\end{equation}
and
\begin{equation}\label{eq:initial_uniform_X_bound_main}
\sup_{0<\varepsilon\leq1}
\|(p_0,u_0,\nabla\phi_0)\|_{X_m^\varepsilon}
\leq
\widetilde C_0,
\end{equation}
where
\[
p_0=(\rho_0)^\gamma.
\]

\begin{theorem}[Uniform regularity]\label{thm:main_uniform_regularity}
Let \(m\geq6\), let \(\Omega\) be a \(C^{m+2}\)-domain, and assume that the
matrix \(A\) in the Navier--slip boundary condition satisfies
\[
A\in C^{m+1}(\partial\Omega).
\]
Let
\[
(p_0,u_0,\nabla\phi_0)
\]
be initial data satisfying
\eqref{eq:initial_density_bound_main}--\eqref{eq:initial_uniform_X_bound_main}
and the compatibility conditions. Then there exist a time
\(T_0>0\) and a constant \(\widetilde C_1>0\), both independent of
\(\varepsilon\in(0,1]\), such that the Navier--Stokes--Poisson problem admits a
unique solution
\[
(\rho^\varepsilon,u^\varepsilon,\nabla\phi^\varepsilon)
\]
on \([0,T_0]\). Moreover, the solution satisfies
\begin{equation}\label{eq:main_uniform_density_bound}
\frac1{2\widehat C_0}
\leq
\rho^\varepsilon(t,x)
\leq
2\widehat C_0,
\qquad
(t,x)\in[0,T_0]\times\Omega,
\end{equation}
and the following uniform estimate:
\begin{equation}\label{eq:main_uniform_estimate_direct}
\begin{aligned}
&\sup_{0\leq t\leq T_0}
\Bigg\{
\|(p^\varepsilon,u^\varepsilon,\nabla\phi^\varepsilon)(t)\|_{\mathcal H^m}^2
+
\|\nabla u^\varepsilon(t)\|_{\mathcal H^{m-1}}^2
+
\sum_{k=0}^{m-2}
\|\partial_t^k\nabla p^\varepsilon(t)\|_{m-1-k}^2 \\
&\qquad
+
\|\Delta p^\varepsilon(t)\|_{\mathcal H^1}^2
+
\|(\nabla u^\varepsilon,\nabla\phi^\varepsilon)(t)\|_{\mathcal H^{1,\infty}}^2
+
\varepsilon
\|\nabla\partial_t^{m-1}p^\varepsilon(t)\|_{L^2}^2 \\
&\qquad
+
\varepsilon
\|\Delta p^\varepsilon(t)\|_{\mathcal H^2}^2
\Bigg\}
+
\int_0^{T_0}
\|\nabla\partial_t^{m-1}p^\varepsilon(t)\|_{L^2}^2\,dt
+
\int_0^{T_0}
\|\Delta p^\varepsilon(t)\|_{\mathcal H^2}^2\,dt \\
&\quad
+
\varepsilon
\int_0^{T_0}
\|\nabla u^\varepsilon(t)\|_{\mathcal H^m}^2\,dt
+
\varepsilon
\sum_{k=0}^{m-2}
\int_0^{T_0}
\|\nabla^2\partial_t^k u^\varepsilon(t)\|_{m-1-k}^2\,dt \\
&\quad
+
\varepsilon^2
\int_0^{T_0}
\|\nabla^2\partial_t^{m-1}u^\varepsilon(t)\|_{L^2}^2\,dt
\leq
\widetilde C_1 .
\end{aligned}
\end{equation}

Here \(\widetilde C_1\) depends only on
\(\widehat C_0\), \(\widetilde C_0\), \(m\), \(\Omega\), and the boundary
matrix \(A\).
\end{theorem}

For \(0<\varepsilon\leq1\), denote by
\[
(\rho^\varepsilon,u^\varepsilon,\nabla\phi^\varepsilon)
\]
the solution on \([0,T_0]\) obtained in
Theorem~\ref{thm:main_uniform_regularity}. With the uniform estimate \eqref{eq:main_uniform_estimate_direct} in hand, one may pass to the limit as \(\varepsilon\to0\) and obtain the following inviscid convergence result.
\begin{theorem}\label{thm:main_inviscid_limit}
Let \(m\geq 6\), let \(\Omega\) be a \(C^{m+2}\)-domain, and assume that
\(A\in C^{m+1}(\partial\Omega)\). Suppose that the initial data
\((p_0,u_0,\nabla\phi_0)\) satisfy the assumptions of
Theorem~\ref{thm:main_uniform_regularity}. For \(0<\varepsilon\leq1\), let
\((\rho^\varepsilon,u^\varepsilon,\nabla\phi^\varepsilon)\) be the solution to
the initial-boundary value problem
\eqref{eq:intro_NSP}--\eqref{eq:intro_phi_boundary} on \([0,T_0]\) with
initial data \((p_0,u_0,\nabla\phi_0)\), satisfying the uniform estimate in
Theorem~\ref{thm:main_uniform_regularity}. Then the Euler--Poisson system
\eqref{eq:intro_EP} with the boundary conditions
\[
u\cdot n=0,\qquad \partial_n\phi=0
\quad\text{on } \partial\Omega
\]
and the initial condition
\[
(\rho,u,\nabla\phi)(x,0)
=
(\rho_0,u_0,\nabla\phi_0)(x),
\qquad x\in\Omega,
\]
admits a unique solution \((\rho,u,\nabla\phi)\) on \([0,T_0]\) satisfying
\[
\sup_{0\leq t\leq T_0}
\|(p,u,\nabla\phi)(t)\|_{X_m^0}<\infty ,
\]
where \(p=\rho^\gamma\). Moreover, it holds that
\[
\begin{aligned}
&\lim_{\varepsilon\to0}
\sup_{0\leq t\leq T_0}
\Bigg(
\|(p^\varepsilon,u^\varepsilon,\nabla\phi^\varepsilon)
-
(p,u,\nabla\phi)\|_{L^2(\Omega)} \\
&\qquad\qquad
+
\|(p^\varepsilon,u^\varepsilon,\nabla\phi^\varepsilon)
-
(p,u,\nabla\phi)\|_{L^\infty(\Omega)}
\Bigg)
=0 .
\end{aligned}
\]

\end{theorem}

\begin{remark}\label{rmk:main_space_difference}
The term
\[
\|\nabla\phi^\varepsilon(t)\|_{\mathcal H^m}^2
\]
is the new component in the conormal energy space compared with the
compressible Navier--Stokes case. It is generated by the electric force
\(\rho^\varepsilon\nabla\phi^\varepsilon\). The Poisson equation and the
Neumann boundary condition allow this term to be included in the energy
functional without introducing additional dissipation.
\end{remark}
\begin{remark}\label{rmk:insulated_bc}
While this paper employs the homogeneous Neumann condition
\eqref{eq:intro_phi_boundary}, the conormal framework can also be adapted to
the insulated boundary setting through the Dirichlet condition
\[
\phi|_{\partial\Omega}=0 .
\]
Under this alternative boundary condition, the normal derivative
\(\partial_n\phi\) is generally nonzero on \(\partial\Omega\). Consequently,
in the basic \(L^2\)-estimate for \(\diver u\), namely
Lemma~\ref{lem:div_gradrho_gradphi}, integration by parts produces the
non-vanishing boundary term
\[
\int_{\partial\Omega}
\rho\,\partial_n\phi\,\diver u\,dS,
\]
which cannot be controlled directly by the low-order energy.

To overcome this difficulty, one may temporarily postpone
Lemma~\ref{lem:div_gradrho_gradphi} and first proceed with the high-order
conormal estimates, Lemma~\ref{lem:high_order_conormal}, together with the
normal derivative estimates. These estimates do not introduce additional
uncontrolled boundary terms, since the conormal derivatives \(Z^\alpha\) are
tangent to the boundary. Once the \(\varepsilon\)-independent estimates for
the pure time derivatives are obtained from the pressure equation,
Lemma~\ref{lem:pure_time_divu_without_epsilon}, sufficient regularity is
available to control the trace of \(\diver u\) on \(\partial\Omega\). With
this high-order information, the above boundary integral can then be estimated
retroactively, and Lemma~\ref{lem:div_gradrho_gradphi} can be closed.
\end{remark}

We now outline the proof of Theorem~\ref{thm:main_uniform_regularity}. The proof
is based on a sequence of uniform a priori estimates in conormal Sobolev spaces.
The basic and high-order conormal energy estimates are first obtained for
\((p^\varepsilon,u^\varepsilon,\nabla\phi^\varepsilon)\). Compared with the
compressible Navier--Stokes system, the main additional difficulty is the electric
force, which is coupled with the density through the Poisson equation. This term
is incorporated into the electrostatic part of the energy by using the continuity
equation and the Poisson equation, while the corresponding commutator terms are
controlled by conormal product estimates and elliptic estimates for the Neumann
problem. Since \(\diver u^\varepsilon\) is not zero, estimates for
\(\diver u^\varepsilon\) and \(\nabla p^\varepsilon\), including the pure time
derivative estimates, are needed to close the high-order energy. The normal
derivative of the velocity is treated through the vorticity form of the
Navier--slip boundary condition: one introduces a boundary vorticity quantity
satisfying a homogeneous Dirichlet boundary condition, and the electric source
terms in its equation are estimated by the previously obtained conormal bounds
and the Poisson elliptic estimates. The \(L^\infty\)-estimates are then derived
from the anisotropic Sobolev embedding, the boundary vorticity estimate, and the
elliptic estimates for the electric potential. Finally, the uniform estimate for
\(\Delta p^\varepsilon\) provides the required control of the Lipschitz norm of
the pressure. Combining these estimates yields the uniform bound in
Theorem~\ref{thm:main_uniform_regularity} by a standard approximation and
continuity argument.

The paper is organized as follows. In Section~\ref{sec:preliminaries}, we
introduce the conormal Sobolev spaces, the local coordinate system near the
boundary, and several basic inequalities which will be used throughout the
paper. Section~\ref{sec:apriori} is devoted to the uniform a priori estimates.
We first derive the basic energy estimates and the high-order conormal
estimates, then establish the normal derivative estimates, the
\(L^\infty\)-estimates, and the uniform estimate for \(\Delta p^\varepsilon\).
Finally, in Section~\ref{sec:proofs}, we complete the proofs of
Theorem~\ref{thm:main_uniform_regularity} and
Theorem~\ref{thm:main_inviscid_limit} by a standard approximation, compactness
and continuity argument.

\section{Preliminaries}\label{sec:preliminaries}

In this section, we collect some notation and inequalities which will be used
throughout the paper. Since boundary layers may appear in the presence of
physical boundaries, it is in general impossible to obtain uniform estimates in
standard Sobolev spaces. Following the conormal framework developed by Wang, Xin and Yong
\cite{WangXinYong2015} for the compressible Navier--Stokes equations with
general Navier--slip boundary conditions, we derive uniform estimates for the
Navier--Stokes--Poisson system. Compared with \cite{WangXinYong2015}, the main
new difficulty comes from the Poisson coupling and the electric force
\(\rho\nabla\phi\), which require uniform elliptic estimates for the electric
field and careful treatment of the interaction between the fluid variables and
the potential.

Let \(\Omega\subset\mathbb R^3\) be a bounded smooth domain. We assume that
\(\Omega\) has a finite covering such that
\begin{equation}\label{eq:domain_covering}
\Omega\subset \Omega_0\cup\bigcup_{k=1}^{n}\Omega_k,
\end{equation}
where \(\overline{\Omega}_0\subset\Omega\), and in each \(\Omega_k\), there
exists a smooth function \(\psi_k\) such that
\[
\Omega\cap\Omega_k
=
\left\{
x=(x_1,x_2,x_3)\mid x_3>\psi_k(x_1,x_2)
\right\}\cap\Omega_k,
\]
and
\[
\partial\Omega\cap\Omega_k
=
\left\{
x=(x_1,x_2,x_3)\mid x_3=\psi_k(x_1,x_2)
\right\}\cap\Omega_k.
\]
We say that \(\Omega\) is of class \(C^m\) if the functions \(\psi_k\) are of
class \(C^m\).

To define the Sobolev conormal spaces, we consider \((Z_k)_{1\leq k\leq N}\),
a finite set of generators of vector fields which are tangent to
\(\partial\Omega\), and set
\[
H^m_{co}
=
\left\{
f\in L^2(\Omega)\mid Z^I f\in L^2(\Omega),\ |I|\leq m
\right\},
\]
where \(I=(k_1,\ldots,k_l)\) is a spatial multi-index. The following notation
will be used:
\[
\|u\|_m^2
=
\|u\|_{H^m_{co}}^2
=
\sum_{j=1}^3
\sum_{|I|\leq m}
\|Z^Iu_j\|_{L^2}^2,
\qquad
\|u\|_{m,\infty}^2
=
\sum_{j=1}^3
\sum_{|I|\leq m}
\|Z^Iu_j\|_{L^\infty}^2,
\]
and
\[
\|\nabla Z^m u\|^2
=
\sum_{|I|=m}
\|\nabla Z^Iu\|_{L^2}^2.
\]

For scalar functions, the same notation is used without the summation over the
components.

Noting that by using the covering of \(\Omega\), one can always assume that
each vector field is supported in one of the \(\Omega_i\). Moreover, in
\(\Omega_0\), the norm \(\|\cdot\|_m\) yields a control of the standard
\(H^m\)-norm, whereas if \(\Omega_i\cap\partial\Omega\neq\emptyset\), there is
no control of the normal derivatives.

Denote by \(C_k\) a positive constant independent of \(\varepsilon\in(0,1]\)
which depends only on the \(C^k\)-norm of the functions \(\psi_j\),
\(j=1,\ldots,n\). Since \(\partial\Omega\) is given locally by
\(x_3=\psi(x_1,x_2)\), we use the coordinates
\[
\Psi:(y,z)\mapsto (y,\psi(y)+z)=x,
\qquad y=(y_1,y_2).
\]
A local basis is given by
\[
e_{y_1}=(1,0,\partial_1\psi)^t,\qquad
e_{y_2}=(0,1,\partial_2\psi)^t,\qquad
e_z=(0,0,-1)^t .
\]
On the boundary, \(e_{y_1}\) and \(e_{y_2}\) are tangent to \(\partial\Omega\).
We take the conormal vector fields
\begin{equation}\label{eq:conormal_vector_fields}
Z_i=\partial_{y_i}=\partial_i+\partial_i\psi\,\partial_z,
\quad i=1,2,
\qquad
Z_3=\varphi(z)\partial_z,
\end{equation}
where
\[
\varphi(z)=\frac{z}{1+z}
\]
is smooth on \(\mathbb R_+\), satisfies \(\varphi(0)=0\),
\(\varphi'(0)>0\), and \(\varphi(z)>0\) for \(z>0\). We also set
\[
Z_0=\partial_t.
\]
It is easy to check that
\[
Z_kZ_j=Z_jZ_k,\qquad j,k=1,2,3,
\]
and
\[
\partial_zZ_i=Z_i\partial_z,\quad i=1,2,
\qquad
\partial_zZ_3\neq Z_3\partial_z .
\]

For a vector field \(u\), its coordinates in the local basis
\((e_{y_1},e_{y_2},e_z)\) are denoted by \(u^i\), namely
\[
u=u^1e_{y_1}+u^2e_{y_2}+u^3e_z.
\]
We use \(u_i\) to denote the coordinates of \(u\) in the standard basis of
\(\mathbb R^3\). The unit outward normal vector is given locally by
\[
n(x)=n(\Psi(y,z))
=
\frac{1}{\sqrt{1+|\nabla\psi(y)|^2}}
\begin{pmatrix}
\partial_1\psi(y)\\
\partial_2\psi(y)\\
-1
\end{pmatrix}.
\]
We also define the orthogonal projection onto the tangent plane by
\[
\Pi u
=
u-(u\cdot n)n .
\]
In the local coordinate patch, both \(n\) and \(\Pi\) are extended smoothly
inside the domain and are independent of \(z\).

For a space-time multi-index
\[
\alpha=(\alpha_0,\alpha_1),
\qquad
\alpha_1=(\alpha_{11},\alpha_{12},\alpha_{13}),
\]
we write
\[
Z^\alpha
=
\partial_t^{\alpha_0}
Z_1^{\alpha_{11}}
Z_2^{\alpha_{12}}
Z_3^{\alpha_{13}},
\qquad
|\alpha|=\alpha_0+\alpha_{11}+\alpha_{12}+\alpha_{13}.
\]
After introducing \(Z_0=\partial_t\), expressions such as
\(\|Z^k f\|_{L^2}^2\) and \(\|\nabla Z^k f\|_{L^2}^2\) denote the sums
\(\sum_{|\alpha|=k}\|Z^\alpha f\|_{L^2}^2\) and
\(\sum_{|\alpha|=k}\|\nabla Z^\alpha f\|_{L^2}^2\), respectively,
over space-time multi-indices. Purely spatial conormal norms retain the
notation \(\|\cdot\|_k\). Expressions such as \(\partial_t^k p_0\) mean
\((\partial_t^k p)|_{t=0}\), with time derivatives evaluated recursively
from the equations; the same convention applies to \(u_0\) and \(\phi_0\).

For a smooth function \(f=f(t,x)\), define
\begin{equation}\label{eq:conormal_norm_def}
\|f(t)\|_{\mathcal H^m}^2
:=
\sum_{|\alpha|\leq m}
\|Z^\alpha f(t)\|_{L^2(\Omega)}^2,
\end{equation}
and
\begin{equation}\label{eq:conormal_Linfty_norm_def}
\|f(t)\|_{\mathcal H^{k,\infty}}^2
:=
\sum_{|\alpha|\leq k}
\|Z^\alpha f(t)\|_{L^\infty(\Omega)}^2.
\end{equation}
Thus, \(\|\cdot\|_m\) only involves the spatial conormal vector fields
\(Z_1,Z_2,Z_3\), while \(\|\cdot\|_{\mathcal H^m}\) also contains the time
derivative \(Z_0=\partial_t\). Throughout this paper, \(\|\cdot\|\) denotes norms over \(\Omega\), while
\(|\cdot|_{H^s(\partial\Omega)}\) denotes the standard Sobolev norm of the
trace on the boundary \(\partial\Omega\).

\begin{proposition}\label{prop:div_curl}
Let \(m\in\mathbb N_+\). Let \(u\in H^m(\Omega)\) be a vector-valued function.
Then there exists a constant \(C>0\), independent of \(u\), such that
\begin{equation}\label{eq:div_curl_normal}
\|u\|_{H^m}
\leq
C\left(
\|\nabla\times u\|_{H^{m-1}}
+
\|\diver u\|_{H^{m-1}}
+
\|u\|_{H^{m-1}}
+
|u\cdot n|_{H^{m-\frac12}(\partial\Omega)}
\right),
\end{equation}
and
\begin{equation}\label{eq:div_curl_tangent}
\|u\|_{H^m}
\leq
C\left(
\|\nabla\times u\|_{H^{m-1}}
+
\|\diver u\|_{H^{m-1}}
+
\|u\|_{H^{m-1}}
+
|n\times u|_{H^{m-\frac12}(\partial\Omega)}
\right).
\end{equation}
\end{proposition}

We shall also use the following conormal product estimate.

\begin{proposition}\label{prop:conormal_product}
For \(u,v\in L^\infty(\Omega)\cap H_{co}^k(\Omega)\) and spatial
multi-indices \(\alpha_1,\alpha_2\), it holds that
\begin{equation}\label{eq:conormal_product}
\|Z^{\alpha_1}u\,Z^{\alpha_2}v\|_{L^2}
\leq
C
\left(
\|u\|_{L^\infty}\|v\|_{H_{co}^k}
+
\|v\|_{L^\infty}\|u\|_{H_{co}^k}
\right),
\qquad
|\alpha_1|+|\alpha_2|=k .
\end{equation}
For space--time conormal derivatives, we use the time-integrated version
of this estimate; see \cite[Proposition~2.2]{WangXinYong2015}. For smooth
\(f,g\), \(|\beta|+|\nu|\leq k\), and \(0<t\leq T\),
\[
\begin{aligned}
\int_0^t\|Z^\beta f\,Z^\nu g\|_{L^2}^2\,d\tau
\leq C\bigg(&\|f\|_{L^\infty([0,t]\times\Omega)}^2
\int_0^t\|g\|_{\mathcal H^k}^2\,d\tau\\
&+\|g\|_{L^\infty([0,t]\times\Omega)}^2
\int_0^t\|f\|_{\mathcal H^k}^2\,d\tau\bigg).
\end{aligned}
\]
\end{proposition}

The following anisotropic Sobolev embedding and trace estimates will be used
frequently.

\begin{proposition}\label{prop:anisotropic_trace}
Let \(m_1\geq0\), \(m_2\geq0\) be integers and
\(f\in H_{\mathrm{co}}^{m_1}(\Omega)\cap H_{\mathrm{co}}^{m_2}(\Omega)\) with
\(\nabla f\in H_{\mathrm{co}}^{m_2}(\Omega)\).

The following anisotropic Sobolev embedding holds:
\begin{equation}\label{eq:anisotropic_sobolev}
\|f\|_{L^\infty}^2
\leq
C
\left(
\|\nabla f\|_{H_{\mathrm{co}}^{m_2}}
+
\|f\|_{H_{\mathrm{co}}^{m_2}}
\right)
\|f\|_{H_{\mathrm{co}}^{m_1}},
\end{equation}
provided that \(m_1+m_2\geq 3\).

The following trace estimate holds:
\begin{equation}\label{eq:trace_estimate}
|f|_{H^s(\partial\Omega)}^2
\leq
C
\left(
\|\nabla f\|_{H_{\mathrm{co}}^{m_2}}
+
\|f\|_{H_{\mathrm{co}}^{m_2}}
\right)
\|f\|_{H_{\mathrm{co}}^{m_1}},
\end{equation}
provided that \(m_1+m_2\geq 2s\geq0\).
\end{proposition}

\begin{proposition}\label{prop:poisson_neumann_estimate}
Assume that \(\Omega\subset\mathbb R^3\) is a bounded smooth domain and that
\(\phi\) satisfies
\[
\Delta\phi=q
\quad\text{in }\Omega,
\qquad
\nabla\phi\cdot n=0
\quad\text{on }\partial\Omega .
\]
Then there exists a positive constant \(C\), depending only on \(\Omega\), such
that
\begin{equation}\label{eq:poisson_H2_estimate}
\|\nabla^2\phi\|_{L^2}
\leq
C
\left(
\|q\|_{L^2}
+
\|\nabla\phi\|_{L^2}
\right).
\end{equation}
Moreover, for the conormal Sobolev space, one has
\begin{equation}\label{eq:poisson_conormal_estimate}
\|\nabla^2\phi\|_{\mathcal H^s}
\leq
C_{s+1}
\left(
\|q\|_{\mathcal H^s}
+
\|\nabla^2\phi\|_{\mathcal H^{s-1}}
+
\|\nabla\phi\|_{L^2}
\right)
\leq
C_{s+1}
\left(
\|q\|_{\mathcal H^s}
+
\|\nabla\phi\|_{L^2}
\right),
\end{equation}
for \(s=1,\ldots,m\).
\end{proposition}

\section{A Priori Estimates}
\label{sec:apriori}

In this section, we derive a priori estimates for smooth solutions
to \eqref{eq:intro_NSP} together with the boundary conditions
\eqref{eq:intro_Navier_boundary} and \eqref{eq:intro_phi_boundary}. For
notational convenience, we drop the superscript \(\varepsilon\) throughout this
section and only keep the viscosity parameter \(\varepsilon\).

By the continuity equation, for any \(0\leq t\leq T\), one has
\[
|\rho_0(x)|
\exp\!\left(-\int_0^t\|\diver u\|_{L^\infty}\,d\tau\right)
\leq
\rho(x,t)
\leq
|\rho_0(x)|
\exp\!\left(\int_0^t\|\diver u\|_{L^\infty}\,d\tau\right).
\]

Throughout this section, we shall work on the interval of time \([0,T]\) such
that
\[
0<c_0\leq \rho(t,x)\leq \frac1{c_0}<\infty,
\qquad (t,x)\in[0,T]\times\Omega .
\]
We point out that the generic constant \(C\) may depend on \(1/c_0\) in this
section.

Let
\[
\omega=\nabla\times u .
\]
Using
\[
\Delta u=\nabla\diver u-\nabla\times\omega ,
\]
the momentum equation can be rewritten as
\begin{equation}\label{eq:momentum_curl_form}
\rho u_t+\rho u\cdot\nabla u+\nabla p
=
-\mu\varepsilon\nabla\times\omega
+
(2\mu+\lambda)\varepsilon\nabla\diver u
+
\rho\nabla\phi .
\end{equation}

\begin{lemma}[Basic energy estimate]\label{lem:basic_energy}
For a smooth solution of \eqref{eq:intro_NSP} subject to
\eqref{eq:intro_Navier_boundary} and \eqref{eq:intro_phi_boundary}, it holds
that, for \(0<\varepsilon\leq1\),
\begin{equation}\label{eq:basic_energy_estimate}
\begin{aligned}
&\sup_{0\leq \tau\leq t}
\left\{
\int_\Omega
\left(
\frac12\rho |u|^2
+
\frac{1}{\gamma-1}\rho^\gamma
\right)(\tau)\,dx
+
\frac12\|\nabla\phi(\tau)\|_{L^2}^2
\right\}
+
c\varepsilon
\int_0^t
\|\nabla u(\tau)\|_{L^2}^2\,d\tau
\\
&\leq
\int_\Omega
\left(
\frac12\rho_0 |u_0|^2
+
\frac{1}{\gamma-1}\rho_0^\gamma
\right)\,dx
+
\frac12\|\nabla\phi_0\|_{L^2}^2
+
C
\int_0^t
\|u(\tau)\|_{L^2}^2\,d\tau ,
\end{aligned}
\end{equation}
where \(c>0\) depends only on \(\mu\) and \(2\mu+\lambda\).
\end{lemma}

\begin{proof}
Multiplying \eqref{eq:momentum_curl_form} by \(u\) and integrating over
\(\Omega\), we obtain
\begin{equation}\label{eq:basic_energy_identity_start}
\begin{aligned}
&\int_\Omega \rho u_t\cdot u\,dx
+
\int_\Omega \rho(u\cdot\nabla u)\cdot u\,dx
+
\int_\Omega \nabla p\cdot u\,dx
\\
={}&
-\mu\varepsilon
\int_\Omega
(\nabla\times\omega)\cdot u\,dx
+
(2\mu+\lambda)\varepsilon
\int_\Omega
\nabla\diver u\cdot u\,dx
+
\int_\Omega
\rho\nabla\phi\cdot u\,dx .
\end{aligned}
\end{equation}
Using the continuity equation and
\eqref{eq:intro_Navier_boundary}, we have
\begin{equation}\label{eq:basic_kinetic_part}
\int_\Omega \rho u_t\cdot u\,dx
+
\int_\Omega \rho(u\cdot\nabla u)\cdot u\,dx
=
\frac12\frac{d}{dt}
\int_\Omega
\rho |u|^2\,dx ,
\end{equation}
and, since \(p=\rho^\gamma\),
\begin{equation}\label{eq:basic_pressure_part}
\int_\Omega \nabla p\cdot u\,dx
=
\frac{d}{dt}
\int_\Omega
\frac{1}{\gamma-1}\rho^\gamma\,dx .
\end{equation}
By \eqref{eq:intro_NSP}, \eqref{eq:intro_Navier_boundary}, and
\eqref{eq:intro_phi_boundary}, we get
\begin{equation}\label{eq:basic_electric_part}
\begin{aligned}
\int_\Omega
\rho\nabla\phi\cdot u\,dx
&=
-\int_\Omega
\phi\,\diver(\rho u)\,dx
=
\int_\Omega
\phi\Delta\phi_t\,dx  =
-\frac12\frac{d}{dt}
\|\nabla\phi\|_{L^2}^2 .
\end{aligned}
\end{equation}
We next estimate the viscous terms. Integrating by parts with the boundary conditions 
\eqref{eq:intro_Navier_boundary} yields
\begin{equation}\label{eq:basic_vorticity_part}
\begin{aligned}
-\varepsilon
\int_\Omega
(\nabla\times\omega)\cdot u\,dx
&=
-\varepsilon\|\omega\|_{L^2}^2
-
\varepsilon
\int_{\partial\Omega}
(n\times\omega)\cdot u\,dS  \\
&\leq
-\varepsilon\|\omega\|_{L^2}^2
+
C\varepsilon
|u|_{L^2(\partial\Omega)}^2 ,
\end{aligned}
\end{equation}
and
\begin{equation}\label{eq:basic_divergence_part}
\varepsilon
\int_\Omega
\nabla\diver u\cdot u\,dx
=
-\varepsilon
\|\diver u\|_{L^2}^2 .
\end{equation}
Combining \eqref{eq:basic_energy_identity_start}--
\eqref{eq:basic_divergence_part}, we obtain

\begin{equation}\label{eq:basic_energy_differential_pre}
\begin{aligned}
&\frac{d}{dt}
\left\{
\int_\Omega
\left(
\frac12\rho |u|^2
+
\frac{1}{\gamma-1}\rho^\gamma
\right)\,dx
+
\frac12\|\nabla\phi\|_{L^2}^2
\right\} \\
&\quad
+
\mu\varepsilon\|\omega\|_{L^2}^2
+
(2\mu+\lambda)\varepsilon\|\diver u\|_{L^2}^2
\leq
C\varepsilon
|u|_{L^2(\partial\Omega)}^2 .
\end{aligned}
\end{equation}

By Proposition~\ref{prop:div_curl} and \eqref{eq:intro_Navier_boundary}, we have
\begin{equation}\label{eq:basic_div_curl_use}
\mu\|\omega\|_{L^2}^2
+
(2\mu+\lambda)\|\diver u\|_{L^2}^2
\geq
2c\|\nabla u\|_{L^2}^2
-
C\|u\|_{L^2}^2 .
\end{equation}
Moreover, Proposition~\ref{prop:anisotropic_trace} gives
\begin{equation}\label{eq:basic_trace_use}
\varepsilon
|u|_{L^2(\partial\Omega)}^2
\leq
C\varepsilon
\left(
\|\nabla u\|_{L^2}
+
\|u\|_{L^2}
\right)
\|u\|_{L^2}
\leq
\frac{c}{2}\varepsilon
\|\nabla u\|_{L^2}^2
+
C\varepsilon
\|u\|_{L^2}^2 .
\end{equation}
Since \(0<\varepsilon\leq1\), it follows from
\eqref{eq:basic_energy_differential_pre}--\eqref{eq:basic_trace_use} that
\begin{equation}\label{eq:basic_energy_differential_final}
\begin{aligned}
&\frac{d}{dt}
\left\{
\int_\Omega
\left(
\frac12\rho |u|^2
+
\frac{1}{\gamma-1}\rho^\gamma
\right)\,dx
+
\frac12\|\nabla\phi\|_{L^2}^2
\right\}
+
c\varepsilon
\|\nabla u\|_{L^2}^2
\leq
C\|u\|_{L^2}^2 .
\end{aligned}
\end{equation}
Integrating \eqref{eq:basic_energy_differential_final} over \([0,t]\) gives
\eqref{eq:basic_energy_estimate}. This completes the proof.
\end{proof}
 We shall derive higher-order conormal estimates
for \((p,u)\), together with the corresponding control of the electric
potential. For later use, set
\begin{equation}\label{eq:Em_def}
\begin{aligned}
\mathcal E_m(t)
:={}&
\|(p,u)(t)\|_{\mathcal H^m}^2
+
\|\nabla u(t)\|_{\mathcal H^{m-1}}^2
+
\sum_{k=0}^{m-2}
\|\nabla\partial_t^k p(t)\|_{m-1-k}^2  \\
&+
\varepsilon
\|\nabla\partial_t^{m-1}p(t)\|_{L^2}^2
+
\|\nabla\phi(t)\|_{\mathcal H^m}^2,
\end{aligned}
\end{equation}
and 
\begin{equation}\label{eq:Q_def}
Q(t)
:=
\sup_{0\leq \tau\leq t}
\left\{
\|(\nabla p,\nabla u)(\tau)\|_{\mathcal H^{1,\infty}}^2
+
\|(p,u,p_t,u_t)(\tau)\|_{L^\infty}^2
+
\|\nabla\phi(\tau)\|_{\mathcal H^{1,\infty}}^2
\right\}.
\end{equation}
Here the norm \(\|\cdot\|_{m-1-k}\) in $\mathcal E_m(t)$ denotes the spatial
conormal norm. This convention is used to emphasize that the estimates for
\(\nabla\partial_t^k p\) require at least one spatial derivative.

\begin{lemma}[Estimate for \(\diver u\) and \(\nabla p\)]
\label{lem:div_gradrho_gradphi}
Let \(m\geq0\). Then it holds that
\begin{equation}\label{eq:div_gradrho_gradphi_estimate}
\begin{aligned}
&\sup_{0\leq \tau\leq t}
\left(
\int_\Omega
\rho|\diver u|^2(\tau)\,dx
+
\int_\Omega
\frac{1}{\gamma p}|\nabla p|^2(\tau)\,dx
\right)
+
c\varepsilon
\int_0^t
\|\nabla\diver u(\tau)\|_{L^2}^2\,d\tau  \\
&\leq
C
\left(
\int_\Omega
\rho_0|\diver u_0|^2\,dx
+
\int_\Omega
\frac{1}{\gamma p_0}|\nabla p_0|^2\,dx
\right)
+
{C_3}[1+P(Q(t))]
\int_0^t
\mathcal E_m(\tau)\,d\tau,
\end{aligned}
\end{equation}
{Here \(P(\cdot)\) denotes a polynomial function, and
\(C_3\) is a constant of the type introduced above; the subscript indicates
that the constant may depend on the geometric coefficients up to order \(3\),
but is independent of \(\varepsilon\) and \(t\).}
\end{lemma}

\begin{proof}
Multiplying \eqref{eq:momentum_curl_form} by \(\nabla\diver u\) and
integrating over \(\Omega\times(0,t)\), we obtain
\begin{equation}\label{eq:divu_gradp_basic_identity}
\begin{aligned}
&\int_0^t\int_\Omega
(\rho u_t+\rho u\cdot\nabla u)\cdot\nabla\diver u\,dx\,d\tau
+
\int_0^t\int_\Omega
\nabla p\cdot\nabla\diver u\,dx\,d\tau  \\
={}&
-\mu\varepsilon
\int_0^t\int_\Omega
(\nabla\times\omega)\cdot\nabla\diver u\,dx\,d\tau
+
(2\mu+\lambda)\varepsilon
\int_0^t
\|\nabla\diver u\|_{L^2}^2\,d\tau  \\
&+
\int_0^t\int_\Omega
\rho\nabla\phi\cdot\nabla\diver u\,dx\,d\tau .
\end{aligned}
\end{equation}

We first estimate the electric force term. Using \(\partial_n\phi=0\) and
\(\Delta\phi=\rho-1\), integration by parts gives
\begin{equation}\label{eq:electric_low_final}
\begin{aligned}
|\mathcal I_E|
&:=
\left|
\int_0^t\int_\Omega
\rho\nabla\phi\cdot\nabla\diver u\,dx\,d\tau
\right|
=
\left|
-\int_0^t\int_\Omega
\diver(\rho\nabla\phi)\,\diver u\,dx\,d\tau
\right|  \\
&=
\left|
-\int_0^t\int_\Omega
\left(
\nabla\rho\cdot\nabla\phi+\rho^2-\rho
\right)\diver u\,dx\,d\tau
\right|
\leq
C[1+P(Q(t))]
\int_0^t
\mathcal E_m(\tau)\,d\tau .
\end{aligned}
\end{equation}

Using integration by parts, the identity
\((u\cdot\nabla)u\cdot n=-u\cdot((u\cdot\nabla)n)\) on \(\partial\Omega\),
and Proposition~\ref{prop:anisotropic_trace}, we obtain
\begin{equation}\label{eq:divu_inertial_estimate}
\begin{aligned}
&\int_0^t\int_\Omega
(\rho u_t+\rho u\cdot\nabla u)\cdot\nabla\diver u\,dx\,d\tau  \\
={}&
-\frac12\int_\Omega \rho|\diver u|^2(t)\,dx
+
\frac12\int_\Omega \rho_0|\diver u_0|^2\,dx\\
&
-
\int_0^t\int_\Omega
\left(
\nabla\rho\cdot u_t
+
\nabla(\rho u)^T:\nabla u
\right)\diver u\,dx\,d\tau  \\
&+
\int_0^t\int_{\partial\Omega}
\rho(u\cdot\nabla)u\cdot n\,\diver u\,dS\,d\tau  \\
\leq{}&
-\frac12\int_\Omega \rho|\diver u|^2(t)\,dx
+
\frac12\int_\Omega \rho_0|\diver u_0|^2\,dx
+
C_2[1+P(Q(t))]
\int_0^t
\mathcal E_m(\tau)\,d\tau .
\end{aligned}
\end{equation}

Since \(p=\rho^\gamma\), we have
\begin{equation}\label{eq:pressure_div_relation}
p_t+u\cdot\nabla p+\gamma p\diver u=0,
\qquad
\diver u=-\frac{p_t+u\cdot\nabla p}{\gamma p}.
\end{equation}
Using \eqref{eq:pressure_div_relation} and integrating by parts, we obtain
\begin{equation}\label{eq:gradp_low}
\begin{aligned}
&\int_0^t\int_\Omega
\nabla p\cdot\nabla\diver u\,dx\,d\tau  \\
&=
-\int_0^t\int_\Omega
\nabla p\cdot\nabla\left(\frac{p_t}{\gamma p}\right)\,dx\,d\tau
-
\int_0^t\int_\Omega
\nabla p\cdot\nabla\left(\frac{u\cdot\nabla p}{\gamma p}\right)\,dx\,d\tau  \\
&\leq
-\int_\Omega
\frac{1}{2\gamma p}|\nabla p|^2(t)\,dx
+
\int_\Omega
\frac{1}{2\gamma p_0}|\nabla p_0|^2\,dx
+
C[1+P(Q(t))]
\int_0^t
\mathcal E_m(\tau)\,d\tau .
\end{aligned}
\end{equation}

For the vorticity term, using
\(\nabla\times\nabla\diver u=0\), \eqref{eq:intro_Navier_boundary},
integration by parts along the boundary, Proposition~\ref{prop:anisotropic_trace},
and Young's inequality, we obtain
\begin{equation}\label{eq:vorticity_low}
\begin{aligned}
&\varepsilon
\left|
\int_0^t\int_\Omega
(\nabla\times\omega)\cdot\nabla\diver u\,dx\,d\tau
\right|
=
\varepsilon
\left|
\int_0^t\int_{\partial\Omega}
(Bu)\cdot \Pi(\nabla\diver u)\,dS\,d\tau
\right|  \\
&=
\varepsilon
\left|
\int_0^t\int_{\partial\Omega}
(Bu)\cdot Z_y\diver u\,dS\,d\tau
\right|
\leq
C_2\varepsilon
\int_0^t
|u|_{H^{\frac12}(\partial\Omega)}
|\diver u|_{H^{\frac12}(\partial\Omega)}\,d\tau  \\
&\leq
\frac{2\mu+\lambda}{4}
\varepsilon
\int_0^t
\|\nabla\diver u(\tau)\|_{L^2}^2\,d\tau
+
C_3\varepsilon
\int_0^t
\left(
\|\nabla u(\tau)\|_{L^2}^2
+
\|u(\tau)\|_{L^2}^2
\right)\,d\tau .
\end{aligned}
\end{equation}

Substituting \eqref{eq:electric_low_final},
\eqref{eq:divu_inertial_estimate}, \eqref{eq:gradp_low}, and
\eqref{eq:vorticity_low} into \eqref{eq:divu_gradp_basic_identity}, we obtain
\eqref{eq:div_gradrho_gradphi_estimate}. This completes the proof.
\end{proof}

\begin{lemma}[Higher-order conormal estimate]\label{lem:high_order_conormal}
Let \(m\geq1\). For any sufficiently small \(\delta>0\), it holds that
\begin{equation}\label{eq:high_order_conormal_estimate}
\begin{aligned}
&\sup_{0\leq \tau\leq t}
\left(
\|(p,u)(\tau)\|_{\mathcal H^m}^2
+
\|\nabla\phi(\tau)\|_{\mathcal H^m}^2
\right)
+
c\varepsilon
\int_0^t
\|\nabla u(\tau)\|_{\mathcal H^m}^2\,d\tau
\\
&\leq
C_{m+2}
\Bigg\{
\mathcal E_m(0)
+
\delta\varepsilon^2
\int_0^t
\|\nabla^2u(\tau)\|_{\mathcal H^{m-1}}^2\,d\tau
+
\delta
\int_0^t
\|\nabla\partial_t^{m-1}p(\tau)\|_{L^2}^2\,d\tau
\\
&\qquad\qquad
+
C_\delta[1+P(Q(t))]
\int_0^t
\mathcal E_m(\tau)\,d\tau
\Bigg\}.
\end{aligned}
\end{equation}
Here \(C_\delta\) denotes a positive constant depending also on the small
parameter \(\delta\), but independent of \(\varepsilon\) and \(t\).
\end{lemma}
\begin{proof}
The case \(m=0\) has been proved in Lemma~\ref{lem:basic_energy}. Assume that
the estimate has been obtained for all orders not larger than
\(m-1\). We prove it for \(|\alpha|=m\).

Applying \(Z^\alpha\), \(|\alpha|=m\), to
\eqref{eq:momentum_curl_form}, multiplying the resulting equation by
\(Z^\alpha u\), and integrating over \(\Omega\), we obtain
\begin{equation}\label{eq:high_momentum_energy_identity}
\begin{aligned}
&\frac12\frac{d}{dt}
\int_\Omega
\rho|Z^\alpha u|^2\,dx
+
\int_\Omega
Z^\alpha\nabla p\cdot Z^\alpha u\,dx
\\
={}&
-\mu\varepsilon
\int_\Omega
Z^\alpha\nabla\times\omega\cdot Z^\alpha u\,dx
+
(2\mu+\lambda)\varepsilon
\int_\Omega
Z^\alpha\nabla\diver u\cdot Z^\alpha u\,dx
\\
&+
\int_\Omega
\rho\nabla Z^\alpha\phi\cdot Z^\alpha u\,dx
+
\int_\Omega
\left(
C_1^\alpha+C_2^\alpha+C_3^\alpha
\right)\cdot Z^\alpha u\,dx ,
\end{aligned}
\end{equation}
where
\begin{equation}\label{eq:high_momentum_commutators}
\begin{aligned}
C_1^\alpha
&:=
-[Z^\alpha,\rho]u_t
=
-\sum_{\substack{|\beta|\geq1\\ \beta+\nu=\alpha}}
C_{\alpha\beta}Z^\beta\rho\,Z^\nu u_t,\\
C_2^\alpha
&:=
-[Z^\alpha,\rho u\cdot\nabla]u
=
-\sum_{\substack{|\beta|\geq1\\ \beta+\nu=\alpha}}
C_{\alpha\beta}Z^\beta(\rho u)\,Z^\nu\nabla u
-
\rho u\cdot[Z^\alpha,\nabla]u,\\
C_3^\alpha
&:=
[Z^\alpha,\rho]\nabla\phi
+
\rho[Z^\alpha,\nabla]\phi .
\end{aligned}
\end{equation}
Here \(C_1^\alpha\), \(C_2^\alpha\) and \(C_3^\alpha\) denote the commutator
terms generated by applying \(Z^\alpha\) to the momentum equation, and
\(C_{\alpha\beta}\) denotes the corresponding combinatorial coefficient.
We first estimate the electric force term. Set
\[
\mathcal J^\alpha
:=
\int_\Omega
\rho\nabla Z^\alpha\phi\cdot Z^\alpha u\,dx .
\]
By integration by parts, the continuity equation, and the Poisson equation, we
obtain
\begin{equation}\label{eq:J_alpha_decomposition}
\begin{aligned}
\mathcal J^\alpha
&=
-\int_\Omega
Z^\alpha\phi\,\diver(\rho Z^\alpha u)\,dx
+
\int_{\partial\Omega}
\rho Z^\alpha\phi\,Z^\alpha u\cdot n\,dS  \\
&=
-\frac12\frac{d}{dt}
\|\nabla Z^\alpha\phi\|_{L^2}^2
+
\int_{\partial\Omega}
\rho Z^\alpha\phi\,Z^\alpha u\cdot n\,dS  \\
&\quad
+
\int_{\partial\Omega}
Z^\alpha\phi\,\partial_n Z^\alpha\phi_t\,dS
+
\int_\Omega
Z^\alpha\phi\,[Z^\alpha,\diver](\rho u)\,dx  \\
&\quad
-
\int_\Omega
\nabla Z^\alpha\phi\cdot[Z^\alpha,\rho]u\,dx
+
\int_{\partial\Omega}
Z^\alpha\phi\,[Z^\alpha,\rho]u\cdot n\,dS  \\
&\quad
+
\int_\Omega
Z^\alpha\phi\,[Z^\alpha,\Delta]\phi_t\,dx  \\
&\triangleq
-\frac12\frac{d}{dt}
\|\nabla Z^\alpha\phi\|_{L^2}^2
+
B_1^\alpha
+
B_2^\alpha
+
R_{\diver}^\alpha
+
R_\rho^\alpha
+
R_\Delta^\alpha .
\end{aligned}
\end{equation}

Here
\[
\begin{aligned}
\diver(\rho Z^\alpha u)
&=
Z^\alpha\diver(\rho u)
-
[Z^\alpha,\diver](\rho u)
-
\diver([Z^\alpha,\rho]u).
\end{aligned}
\]

We estimate \(B_1^\alpha\). If \(Z^\alpha\) is purely temporal, then
\(Z^\alpha u\cdot n=0\) on \(\partial\Omega\). Otherwise, differentiating
\(u\cdot n=0\) tangentially on \(\partial\Omega\) gives
\[
|Z^\alpha u\cdot n|_{L^2(\partial\Omega)}
\leq
C_{m+1}|u|_{H^{m-1}(\partial\Omega)} .
\]
Moreover,
\[
|Z^\alpha\phi|_{L^2(\partial\Omega)}
\leq
C|\nabla\phi|_{H^{m-1}(\partial\Omega)} .
\]
Thus, by Proposition~\ref{prop:anisotropic_trace},
\[
\begin{aligned}
|B_1^\alpha|
&\leq
C_{m+1}
|\nabla\phi|_{H^{m-1}(\partial\Omega)}
|u|_{H^{m-1}(\partial\Omega)}  \\
&\leq
C_{m+2}
\left(
\|\nabla^2\phi\|_{\mathcal H^{m-1}}
+
\|\nabla\phi\|_{\mathcal H^{m-1}}
\right)^{1/2}
\|\nabla\phi\|_{\mathcal H^{m-1}}^{1/2}
\left(
\|\nabla u\|_{\mathcal H^{m-1}}
+
\|u\|_{\mathcal H^m}
\right).
\end{aligned}
\]
By Proposition~\ref{prop:poisson_neumann_estimate},
\begin{equation}\label{eq:phi_elliptic_control}
\|\nabla^2\phi\|_{\mathcal H^{m-1}}
\leq
C_{m+2}
\left(
\|\rho-1\|_{\mathcal H^{m-1}}
+
\|\nabla\phi\|_{L^2}
\right).
\end{equation}
Therefore,
\begin{equation}\label{eq:B1_alpha_estimate}
\sum_{|\alpha|=m}
\int_0^t
|B_1^\alpha(\tau)|\,d\tau
\leq
C_{m+2}[1+P(Q(t))]
\int_0^t
\mathcal E_m(\tau)\,d\tau .
\end{equation}

The term \(B_2^\alpha\) is treated similarly. If \(Z^\alpha\) is purely
temporal, then \(\partial_nZ^\alpha\phi_t=0\) on \(\partial\Omega\). Otherwise,
\[
\partial_nZ^\alpha\phi_t
=
[\partial_n,Z^\alpha]\phi_t
\quad\text{on }\partial\Omega .
\]
Applying Proposition~\ref{prop:anisotropic_trace} as above, and using
Proposition~\ref{prop:poisson_neumann_estimate} for the time-differentiated
Poisson equation, we get
\begin{equation}\label{eq:phit_elliptic_control}
\|\nabla^2\phi_t\|_{\mathcal H^{m-1}}
\leq
C_{m+2}
\left(
\|\rho_t\|_{\mathcal H^{m-1}}
+
\|\nabla\phi_t\|_{L^2}
\right).
\end{equation}
Consequently, as in \eqref{eq:B1_alpha_estimate},
\begin{equation}\label{eq:B2_alpha_estimate}
\sum_{|\alpha|=m}
\int_0^t
|B_2^\alpha(\tau)|\,d\tau
\leq
C_{m+2}[1+P(Q(t))]
\int_0^t
\mathcal E_m(\tau)\,d\tau .
\end{equation}

For \(R_{\diver}^\alpha\), the
commutator \([Z^\alpha,\diver]\) contains only differentiated coefficients of
the conormal vector fields and lower-order derivatives of \(\rho u\). Moreover,
it vanishes when \(Z^\alpha\) is purely temporal. Hence, by
Proposition~\ref{prop:conormal_product},
\begin{equation}\label{eq:Rdiv_alpha_estimate}
\begin{aligned}
\sum_{|\alpha|=m}
\int_0^t
|R_{\diver}^\alpha(\tau)|\,d\tau
&\leq
\sum_{|\alpha|=m}
\int_0^t
\|Z^\alpha\phi\|_{L^2}
\|[Z^\alpha,\diver](\rho u)\|_{L^2}\,d\tau  \\
&\leq
C_{m+1}[1+P(Q(t))]
\int_0^t
\mathcal E_m(\tau)\,d\tau .
\end{aligned}
\end{equation}

For \(R_\rho^\alpha\), the integration by parts in
\eqref{eq:J_alpha_decomposition} gives
\[
R_\rho^\alpha
=
-\int_\Omega
\nabla Z^\alpha\phi\cdot [Z^\alpha,\rho]u\,dx
+
\int_{\partial\Omega}
Z^\alpha\phi\,[Z^\alpha,\rho]u\cdot n\,dS .
\]
Proposition~\ref{prop:conormal_product} and
Proposition~\ref{prop:anisotropic_trace} imply
\begin{equation}\label{eq:Rrho_alpha_estimate}
\begin{aligned}
\sum_{|\alpha|=m}
\int_0^t
|R_\rho^\alpha(\tau)|\,d\tau
&\leq
\sum_{|\alpha|=m}
\int_0^t
\|\nabla Z^\alpha\phi\|_{L^2}
\|[Z^\alpha,\rho]u\|_{L^2}\,d\tau  \\
&\quad+
\sum_{|\alpha|=m}
\int_0^t
|Z^\alpha\phi|_{L^2(\partial\Omega)}
|[Z^\alpha,\rho]u\cdot n|_{L^2(\partial\Omega)}\,d\tau  \\
&\leq
C_{m+2}[1+P(Q(t))]
\int_0^t
\mathcal E_m(\tau)\,d\tau .
\end{aligned}
\end{equation}

Finally, for \(R_\Delta^\alpha\), the commutator
\([Z^\alpha,\Delta]\phi_t\) satisfies
\[
\|[Z^\alpha,\Delta]\phi_t\|_{L^2}
\leq
C_{m+2}
\left(
\|\nabla^2\phi_t\|_{\mathcal H^{m-1}}
+
\|\nabla\phi_t\|_{\mathcal H^{m-1}}
\right).
\]
Thus, by Proposition~\ref{prop:conormal_product},
\eqref{eq:phit_elliptic_control},
\begin{equation}\label{eq:Rdelta_alpha_estimate}
\begin{aligned}
\sum_{|\alpha|=m}
\int_0^t
|R_\Delta^\alpha(\tau)|\,d\tau
&\leq
\sum_{|\alpha|=m}
\int_0^t
\|Z^\alpha\phi\|_{L^2}
\|[Z^\alpha,\Delta]\phi_t\|_{L^2}\,d\tau  \\
&\leq
C_{m+2}[1+P(Q(t))]
\int_0^t
\mathcal E_m(\tau)\,d\tau .
\end{aligned}
\end{equation}

Combining \eqref{eq:J_alpha_decomposition},
\eqref{eq:B1_alpha_estimate}, \eqref{eq:B2_alpha_estimate},
\eqref{eq:Rdiv_alpha_estimate}, \eqref{eq:Rrho_alpha_estimate}, and
\eqref{eq:Rdelta_alpha_estimate}, we obtain
\begin{equation}\label{eq:electric_principal_final}
\begin{aligned}
\sum_{|\alpha|=m}
\int_0^t
\mathcal J^\alpha(\tau)\,d\tau
\leq{}&
-\frac12
\|\nabla Z^m\phi(t)\|_{L^2}^2
+
\frac12
\|\nabla Z^m\phi_0\|_{L^2}^2   \\
&\quad
+
C_{m+2}[1+P(Q(t))]
\int_0^t
\mathcal E_m(\tau)\,d\tau .
\end{aligned}
\end{equation}

The viscous terms are considered as in Wang, Xin and Yong \cite{WangXinYong2015},
Lemma 3.4. More precisely, after integration by parts, using
\eqref{eq:intro_Navier_boundary}, Proposition~\ref{prop:div_curl},
Proposition~\ref{prop:anisotropic_trace}, and Young's inequality, one obtains
\begin{equation}\label{eq:high_viscous_final}
\begin{aligned}
&-\mu\varepsilon
\int_\Omega
Z^\alpha\nabla\times\omega\cdot Z^\alpha u\,dx
+
(2\mu+\lambda)\varepsilon
\int_\Omega
Z^\alpha\nabla\diver u\cdot Z^\alpha u\,dx
\\
&\leq
-c\varepsilon\|\nabla Z^\alpha u\|_{L^2}^2
+
\delta\varepsilon
\|\nabla u\|_{\mathcal H^m}^2
+
\delta\varepsilon^2
\|\nabla^2u\|_{\mathcal H^{m-1}}^2
\\
&\quad
+
C_\delta C_{m+2}
\left(
\|\nabla u\|_{\mathcal H^{m-1}}^2
+
\|u\|_{\mathcal H^m}^2
\right).
\end{aligned}
\end{equation}

We next estimate the pressure term. Integration by parts gives
\begin{equation}\label{eq:high_pressure_integration}
\begin{aligned}
\int_\Omega
Z^\alpha\nabla p\cdot Z^\alpha u\,dx
={}&
-\int_\Omega
Z^\alpha p\,Z^\alpha\diver u\,dx
+
\int_{\partial\Omega}
Z^\alpha p\,Z^\alpha u\cdot n\,dS  \\
&+
\int_\Omega
[Z^\alpha,\nabla]p\cdot Z^\alpha u\,dx .
\end{aligned}
\end{equation}
For the last two terms, if \(Z^\alpha\) is purely temporal, then
\(Z^\alpha u\cdot n=0\) on \(\partial\Omega\). Otherwise, writing the boundary
conormal derivative as \(Z_y^{\alpha_1}\partial_t^{\alpha_0}\), integration by
parts along \(\partial\Omega\) yields
\[
\begin{aligned}
\left|
\int_{\partial\Omega}
Z^\alpha p\,Z^\alpha u\cdot n\,dS
\right|
&\leq
C_{m+2}
\left(
\|\nabla Z_y^{\alpha_1-1}\partial_t^{\alpha_0}p\|_{L^2}
+
\|Z_y^{\alpha_1-1}\partial_t^{\alpha_0}p\|_{\mathcal H^1}
\right)^{1/2}  \\
&\quad \times
\|Z_y^{\alpha_1-1}\partial_t^{\alpha_0}p\|_{\mathcal H^1}^{1/2}
\left(
\|\nabla u\|_{\mathcal H^{m-1}}
+
\|u\|_{\mathcal H^m}
\right).
\end{aligned}
\]
Hence, together with Proposition~\ref{prop:anisotropic_trace},
\begin{equation}\label{eq:high_pressure_boundary_commutator}
\begin{aligned}
&\left|
\int_{\partial\Omega}
Z^\alpha p\,Z^\alpha u\cdot n\,dS
+
\int_\Omega
[Z^\alpha,\nabla]p\cdot Z^\alpha u\,dx
\right| \\
&\leq
\delta
\|\nabla p\|_{\mathcal H^{m-1}}^2
+
C_\delta C_{m+2}
\left(
\|\nabla u\|_{\mathcal H^{m-1}}^2
+
\|(p,u)\|_{\mathcal H^m}^2
\right).
\end{aligned}
\end{equation}

It remains to estimate the first term on the right-hand side of
\eqref{eq:high_pressure_integration}. Applying \(Z^\alpha\) to
\eqref{eq:pressure_div_relation}, we obtain
\begin{equation}\label{eq:Zalpha_divu_pressure_relation}
\begin{aligned}
Z^\alpha\diver u
={}&
-\frac{1}{\gamma p}Z^\alpha p_t
-
\frac{u}{\gamma p}\cdot Z^\alpha\nabla p  \\
&-
\sum_{\substack{|\beta|\geq1\\ \beta+\nu=\alpha}}
C_{\alpha\beta}
Z^\beta\left(\frac{1}{\gamma p}\right)Z^\nu p_t
-
\sum_{\substack{|\beta|\geq1\\ \beta+\nu=\alpha}}
C_{\alpha\beta}
Z^\beta\left(\frac{u}{\gamma p}\right)\cdot Z^\nu\nabla p .
\end{aligned}
\end{equation}
Substituting \eqref{eq:Zalpha_divu_pressure_relation} gives
\begin{equation}\label{eq:pressure_main_substitution}
\begin{aligned}
-\int_0^t\int_\Omega
Z^\alpha p\,Z^\alpha\diver u\,dx\,d\tau
={}&
\int_0^t\int_\Omega
\frac{1}{\gamma p}Z^\alpha p\,Z^\alpha p_t\,dx\,d\tau  \\
&+
\int_0^t\int_\Omega
Z^\alpha p\,\frac{u}{\gamma p}\cdot Z^\alpha\nabla p\,dx\,d\tau
+
R_\alpha^p .
\end{aligned}
\end{equation}
For the first two terms,
\begin{equation}\label{eq:pressure_first_two_terms}
\begin{aligned}
&\int_0^t\int_\Omega
\frac{1}{\gamma p}Z^\alpha p\,Z^\alpha p_t\,dx\,d\tau
+
\int_0^t\int_\Omega
Z^\alpha p\,\frac{u}{\gamma p}\cdot Z^\alpha\nabla p\,dx\,d\tau \\
&\geq
\int_\Omega
\frac{1}{2\gamma p}|Z^\alpha p|^2(t)\,dx
-
\int_\Omega
\frac{1}{2\gamma p_0}|Z^\alpha p_0|^2\,dx
-
\delta
\int_0^t
\|\nabla p(\tau)\|_{\mathcal H^{m-1}}^2\,d\tau  \\
&\quad
-
C_\delta[1+P(Q(t))]
\int_0^t
\|p(\tau)\|_{\mathcal H^m}^2\,d\tau .
\end{aligned}
\end{equation}
The remaining terms in \(R_\alpha^p\) contain at least one differentiated
coefficient. By Proposition~\ref{prop:conormal_product},
\begin{equation}\label{eq:pressure_commutator_remainder}
|R_\alpha^p|
\leq
\delta
\int_0^t
\|\nabla p(\tau)\|_{\mathcal H^{m-1}}^2\,d\tau
+
C_\delta[1+P(Q(t))]
\int_0^t
\|(p,u)(\tau)\|_{\mathcal H^m}^2\,d\tau .
\end{equation}
Combining \eqref{eq:high_pressure_integration}--
\eqref{eq:pressure_commutator_remainder}, we arrive at
\begin{equation}\label{eq:high_pressure_final}
\begin{aligned}
&\int_0^t\int_\Omega
Z^\alpha\nabla p\cdot Z^\alpha u\,dx\,d\tau  \\
&\geq
\int_\Omega
\frac{1}{2\gamma p}|Z^\alpha p|^2(t)\,dx
-
\int_\Omega
\frac{1}{2\gamma p_0}|Z^\alpha p_0|^2\,dx
-
\delta
\int_0^t
\|\nabla p(\tau)\|_{\mathcal H^{m-1}}^2\,d\tau  \\
&\quad
-
C_\delta C_{m+2}[1+P(Q(t))]
\int_0^t
\mathcal E_m(\tau)\,d\tau .
\end{aligned}
\end{equation}

The commutators \(C_1^\alpha,C_2^\alpha\) and \(C_3^\alpha\) are estimated in
the same way. By Proposition~\ref{prop:conormal_product},
\begin{equation}\label{eq:C1_C2_C3_estimate}
\sum_{|\alpha|=m}
\int_0^t
\left|
\int_\Omega
(C_1^\alpha+C_2^\alpha+C_3^\alpha)\cdot Z^\alpha u\,dx
\right|d\tau
\leq
C[1+P(Q(t))]
\int_0^t
\mathcal E_m(\tau)\,d\tau .
\end{equation}

Substituting \eqref{eq:electric_principal_final},
\eqref{eq:high_viscous_final}, \eqref{eq:high_pressure_final}, and
\eqref{eq:C1_C2_C3_estimate} into
\eqref{eq:high_momentum_energy_identity}, summing over \(|\alpha|=m\), and
integrating over \([0,t]\), we obtain
\[
\begin{aligned}
&\sup_{0\leq \tau\leq t}
\left(
\|Z^m(p,u)(\tau)\|_{L^2}^2
+
\|\nabla Z^m\phi(\tau)\|_{L^2}^2
\right)
+
\varepsilon
\int_0^t
\|\nabla Z^m u(\tau)\|_{L^2}^2\,d\tau
\\
&\leq
C_{m+2}
\Bigg\{
\mathcal E_m(0)
+
\delta\varepsilon
\int_0^t
\|\nabla u(\tau)\|_{\mathcal H^m}^2\,d\tau
+
\delta\varepsilon^2
\int_0^t
\|\nabla^2u(\tau)\|_{\mathcal H^{m-1}}^2\,d\tau
\\
&\qquad
+
\delta
\int_0^t
\|\nabla p(\tau)\|_{\mathcal H^{m-1}}^2\,d\tau
+
C_\delta[1+P(Q(t))]
\int_0^t
\mathcal E_m(\tau)\,d\tau
\Bigg\}.
\end{aligned}
\]
Taking \(\delta>0\) sufficiently small and using the induction hypothesis for
lower conormal orders yield \eqref{eq:high_order_conormal_estimate}. This
completes the proof.
\end{proof}

\begin{lemma}[High-order estimate for \(\diver u\) and \(\nabla p\)]
\label{lem:divu_high_conormal}
Let \(m\geq1\), and let \(Z^\alpha\) be a conormal differential operator such
that \(|\alpha|=m-1\) and \(\alpha_0\leq m-2\). Then, for any sufficiently small
\(\delta>0\), it holds that
\begin{equation}\label{eq:divu_high_conormal_estimate}
\begin{aligned}
&\sup_{0\leq \tau\leq t}
\int_\Omega
\left(
\rho |Z^\alpha\diver u|^2
+
\frac{1}{\gamma p}|Z^\alpha\nabla p|^2
\right)(\tau)\,dx
+
c\varepsilon
\int_0^t
\|\nabla Z^\alpha\diver u(\tau)\|_{L^2}^2\,d\tau
\\
&\leq
C
\int_\Omega
\left(
\rho_0 |Z^\alpha\diver u_0|^2
+
\frac{1}{\gamma p_0}|Z^\alpha\nabla p_0|^2
\right)\,dx
\\
&\quad
+
C_{m+2}
\Bigg\{
\delta
\int_0^t
\|\nabla\partial_t^{m-1}p(\tau)\|_{L^2}^2\,d\tau
+
(\delta+\varepsilon)
\int_0^t
\|\nabla Z^{m-2}\diver u(\tau)\|_{L^2}^2\,d\tau
\\
&\qquad\qquad
+
\varepsilon
\int_0^t
\|\nabla^2u(\tau)\|_{\mathcal H^{m-2}}^2\,d\tau
+
C_\delta[1+P(Q(t))]
\int_0^t
\mathcal E_m(\tau)\,d\tau
\Bigg\}.
\end{aligned}
\end{equation}
\end{lemma}

\begin{proof}
We argue by induction on the order. The case \(|\alpha|=0\) follows
from Lemma~\ref{lem:div_gradrho_gradphi}. Assume that the desired estimate has
been proved for all orders not larger than \(m-2\). We prove it for
\(|\alpha|=m-1\), where \(\alpha_0\leq m-2\).

Applying \(Z^\alpha\) to \eqref{eq:momentum_curl_form}, multiplying the
resulting equation by \(\nabla Z^\alpha\diver u\), and integrating over
\(\Omega\times(0,t)\), we obtain
\begin{equation}\label{eq:divu_high_identity}
\begin{aligned}
&\int_0^t\!\!\int_\Omega
\left(
\rho Z^\alpha u_t
+
\rho u\cdot\nabla Z^\alpha u
+
Z^\alpha\nabla p
\right)
\cdot\nabla Z^\alpha\diver u\,dx\,d\tau
\\
={}&
-\mu\varepsilon
\int_0^t\!\!\int_\Omega
Z^\alpha\nabla\times\omega\cdot\nabla Z^\alpha\diver u\,dx\,d\tau
+
(2\mu+\lambda)\varepsilon
\int_0^t
\|\nabla Z^\alpha\diver u\|_{L^2}^2\,d\tau
\\
&+
\int_0^t\!\!\int_\Omega
Z^\alpha(\rho\nabla\phi)\cdot\nabla Z^\alpha\diver u\,dx\,d\tau
+
\int_0^t\!\!\int_\Omega
(D_1^\alpha+D_2^\alpha)\cdot\nabla Z^\alpha\diver u\,dx\,d\tau .
\end{aligned}
\end{equation}

where
\[
D_1^\alpha
:=
-[Z^\alpha,\rho]u_t,
\qquad
D_2^\alpha
:=
-[Z^\alpha,\rho u\cdot\nabla]u .
\]

We first estimate the electric force term. We write
\begin{equation}\label{eq:electric_divu_decomposition}
\begin{aligned}
\mathcal I_E^\alpha
&:=
\int_0^t\int_\Omega
Z^\alpha(\rho\nabla\phi)\cdot\nabla Z^\alpha\diver u\,dx\,d\tau  \\
&=
\int_0^t\int_\Omega
\rho\nabla Z^\alpha\phi\cdot\nabla Z^\alpha\diver u\,dx\,d\tau
+
\int_0^t\int_\Omega
[Z^\alpha,\rho]\nabla\phi\cdot\nabla Z^\alpha\diver u\,dx\,d\tau  \\
&\quad
+
\int_0^t\int_\Omega
\rho[Z^\alpha,\nabla]\phi\cdot\nabla Z^\alpha\diver u\,dx\,d\tau  \\
&\triangleq
I_1^\alpha+I_2^\alpha+I_3^\alpha .
\end{aligned}
\end{equation}

For \(I_1^\alpha\), integration by parts gives
\begin{equation}\label{eq:IE1_high_decomposition}
\begin{aligned}
I_1^\alpha
&=
-\int_0^t\int_\Omega
\diver(\rho\nabla Z^\alpha\phi)\,
Z^\alpha\diver u\,dx\,d\tau
+
\int_0^t\int_{\partial\Omega}
\rho\partial_nZ^\alpha\phi\,
Z^\alpha\diver u\,dS\,d\tau
\\
&\triangleq
I_{1,\mathrm{int}}^\alpha
+
I_{1,\mathrm{bd}}^\alpha .
\end{aligned}
\end{equation}
Since
\[
\diver(\rho\nabla Z^\alpha\phi)
=
\nabla\rho\cdot\nabla Z^\alpha\phi
+
\rho Z^\alpha(\rho-1)
+
\rho[\Delta,Z^\alpha]\phi ,
\]
the first two terms are estimated by Proposition~\ref{prop:conormal_product}
and Proposition~\ref{prop:poisson_neumann_estimate}. The last one is treated in the same way as in
\eqref{eq:Rdelta_alpha_estimate}. Hence
\begin{equation}\label{eq:IE1_interior_estimate}
|I_{1,\mathrm{int}}^\alpha|
\leq
C_{m+2}[1+P(Q(t))]
\int_0^t
\mathcal E_m(\tau)\,d\tau .
\end{equation}
For the boundary part, using \eqref{eq:intro_phi_boundary}, integration by
parts along \(\partial\Omega\), Proposition~\ref{prop:anisotropic_trace}, and
Young's inequality, we get
\begin{equation}\label{eq:IE1_boundary_estimate}
|I_{1,\mathrm{bd}}^\alpha|
\leq
\delta
\int_0^t
\|\nabla Z^{m-2}\diver u(\tau)\|_{L^2}^2\,d\tau
+
C_\delta C_{m+2}[1+P(Q(t))]
\int_0^t
\mathcal E_m(\tau)\,d\tau .
\end{equation}

For \(I_2^\alpha\), integration by parts gives
\[
I_2^\alpha
=-\int_0^t\!\!\int_\Omega
\diver([Z^\alpha,\rho]\nabla\phi)\,Z^\alpha\diver u\,dx\,d\tau
+\int_0^t\!\!\int_{\partial\Omega}
h_\alpha Z^\alpha\diver u\,dS\,d\tau,
\qquad
h_\alpha:=n\cdot[Z^\alpha,\rho]\nabla\phi .
\]
The boundary term is not identically zero on a curved boundary. Indeed,
for a tangential vector field \(Z_i\), the Neumann condition implies
\(n\cdot Z_i\nabla\phi=-(Z_i n)\cdot\nabla\phi\).
More generally, if \(Z^\alpha\) contains no \(Z_3\), then on the boundary
\[
h_\alpha
=-\sum_{0<\beta\leq\alpha}\binom{\alpha}{\beta}
(Z^\beta\rho)
\sum_{0<\nu\leq\alpha-\beta}\binom{\alpha-\beta}{\nu}
(Z^\nu n)\cdot Z^{\alpha-\beta-\nu}\nabla\phi .
\]
In particular, each nonzero term contains a derivative of the normal and
at most \(m-2\) conormal derivatives of the unknowns. Since
\(\alpha_0\leq m-2\), there is a spatial tangential derivative \(T\) in
\(Z^\alpha\). Write \(Z^\alpha\diver u=Tg\), where \(g\) has at most
\(m-2\) conormal derivatives of \(\diver u\). Tangential integration by
parts yields
\[
\int_{\partial\Omega}h_\alpha Tg\,dS
=-\int_{\partial\Omega}
\big(Th_\alpha+(\diver_{\partial\Omega}T)h_\alpha\big)g\,dS.
\]
The local cut-off terms have the same lower order. More explicitly, the
product, trace, and Poisson estimates give
\[
\begin{aligned}
&\int_0^t
|Th_\alpha+(\diver_{\partial\Omega}T)h_\alpha|_{L^2(\partial\Omega)}^2\,d\tau
\leq C_{m+2}[1+P(Q(t))]\int_0^t\mathcal E_m\,d\tau,\\
&\int_0^t|g|_{L^2(\partial\Omega)}^2\,d\tau
\leq C\int_0^t\mathcal E_m\,d\tau
+C\left(\int_0^t\mathcal E_m\,d\tau\right)^{1/2}
\left(\int_0^t\|\nabla Z^{m-2}\diver u\|_{L^2}^2\,d\tau\right)^{1/2}.
\end{aligned}
\]
Here the first estimate uses only mixed pressure gradients with at most
\(m-2\) time derivatives; these are contained in \(\mathcal E_m\).
The second is the trace inequality
\(|g|_{L^2(\partial\Omega)}^2
\leq C\|g\|_{L^2}(\|g\|_{L^2}+\|\nabla g\|_{L^2})\).
If \(Z^\alpha\) contains \(Z_3\), its trace is zero and there is no boundary
contribution. Consequently, estimating the interior term as above and
applying Young's inequality to the boundary term, we obtain
\begin{equation}\label{eq:IE2_estimate}
\begin{aligned}
|I_2^\alpha|
\leq{}&\delta\int_0^t
\|\nabla Z^{m-2}\diver u\|_{L^2}^2\,d\tau\\
&+C_\delta C_{m+2}[1+P(Q(t))]
\int_0^t\mathcal E_m(\tau)\,d\tau .
\end{aligned}
\end{equation}

For \(I_3^\alpha\), integration by parts gives
\begin{equation}\label{eq:IE3_high_decomposition}
\begin{aligned}
I_3^\alpha=
-\int_0^t\int_\Omega
\diver\left(\rho[Z^\alpha,\nabla]\phi\right)
Z^\alpha\diver u\,dx\,d\tau \\
\quad
+
\int_0^t\int_{\partial\Omega}
\rho[Z^\alpha,\nabla]\phi\cdot n\,
Z^\alpha\diver u\,dS\,d\tau .
\end{aligned}
\end{equation}
Using Proposition~\ref{prop:conormal_product} for the interior part and
Proposition~\ref{prop:anisotropic_trace} for the boundary part, together with
Proposition~\ref{prop:poisson_neumann_estimate}, we obtain
\begin{equation}\label{eq:IE3_estimate}
|I_3^\alpha|
\leq
\delta
\int_0^t
\|\nabla Z^{m-2}\diver u(\tau)\|_{L^2}^2\,d\tau
+
C_\delta C_{m+2}[1+P(Q(t))]
\int_0^t
\mathcal E_m(\tau)\,d\tau .
\end{equation}
Combining \eqref{eq:IE1_interior_estimate}--
\eqref{eq:IE3_estimate}, we obtain
\begin{equation}\label{eq:electric_divu_high_final}
|\mathcal I_E^\alpha|
\leq
\delta
\int_0^t
\|\nabla Z^{m-2}\diver u(\tau)\|_{L^2}^2\,d\tau
+
C_\delta C_{m+2}[1+P(Q(t))]
\int_0^t
\mathcal E_m(\tau)\,d\tau .
\end{equation}

We next consider the inertial part. Integrating by parts, using the continuity
equation and \eqref{eq:intro_Navier_boundary}, and estimating the boundary
contribution by Proposition~\ref{prop:anisotropic_trace}, we get
\begin{equation}\label{eq:divu_high_inertial}
\begin{aligned}
&\int_0^t\int_\Omega
\left(
\rho Z^\alpha u_t
+
\rho u\cdot\nabla Z^\alpha u
\right)
\cdot\nabla Z^\alpha\diver u\,dx\,d\tau
\\
&\leq
-\frac12
\int_\Omega
\rho |Z^\alpha\diver u|^2(t)\,dx
+
\frac12
\int_\Omega
\rho_0 |Z^\alpha\diver u_0|^2\,dx
\\
&\quad
+
\delta
\int_0^t
\|\nabla Z^{m-2}\diver u(\tau)\|_{L^2}^2\,d\tau
+
C_\delta C_{m+2}[1+P(Q(t))]
\int_0^t
\mathcal E_m(\tau)\,d\tau .
\end{aligned}
\end{equation}

Next, we estimate the pressure term. Applying \(\nabla\) to
\eqref{eq:Zalpha_divu_pressure_relation} and arguing as in
\eqref{eq:pressure_main_substitution}--\eqref{eq:high_pressure_final}, we obtain,
after integration by parts in time and space,
\begin{equation}\label{eq:divu_high_pressure}
\begin{aligned}
&\int_0^t\int_\Omega
Z^\alpha\nabla p\cdot\nabla Z^\alpha\diver u\,dx\,d\tau
\\
&\leq
-\int_\Omega
\frac{1}{2\gamma p}|Z^\alpha\nabla p|^2(t)\,dx
+
\int_\Omega
\frac{1}{2\gamma p_0}|Z^\alpha\nabla p_0|^2\,dx
\\
&\quad
+
\delta
\int_0^t
\|\nabla\partial_t^{m-1}p(\tau)\|_{L^2}^2\,d\tau
+
C_\delta[1+P(Q(t))]
\int_0^t
\mathcal E_m(\tau)\,d\tau .
\end{aligned}
\end{equation}
The commutator terms generated by differentiating the coefficients
\((\gamma p)^{-1}\) and \(u(\gamma p)^{-1}\) are estimated by
Proposition~\ref{prop:conormal_product}, as in the pressure commutator estimate
of Lemma~\ref{lem:high_order_conormal}.

For the viscous terms, we use
\(\nabla\times\nabla Z^\alpha\diver u=0\), \eqref{eq:intro_Navier_boundary},
integration by parts along \(\partial\Omega\),
Proposition~\ref{prop:anisotropic_trace}, and Young's inequality. This yields
\begin{equation}\label{eq:divu_high_viscosity}
\begin{aligned}
&-\mu\varepsilon
\int_0^t\int_\Omega
Z^\alpha\nabla\times\omega\cdot\nabla Z^\alpha\diver u\,dx\,d\tau
+
(2\mu+\lambda)\varepsilon
\int_0^t
\|\nabla Z^\alpha\diver u\|_{L^2}^2\,d\tau
\\
&\geq
c\varepsilon
\int_0^t
\|\nabla Z^\alpha\diver u(\tau)\|_{L^2}^2\,d\tau
-
C\varepsilon
\int_0^t
\|\nabla Z^{m-2}\diver u(\tau)\|_{L^2}^2\,d\tau
\\
&\quad
-
C\varepsilon
\int_0^t
\|\nabla^2u(\tau)\|_{\mathcal H^{m-2}}^2\,d\tau
-
C_{m+2}[1+P(Q(t))]
\int_0^t
\mathcal E_m(\tau)\,d\tau .
\end{aligned}
\end{equation}

The commutators \(D_1^\alpha\) and \(D_2^\alpha\) satisfy, by
Proposition~\ref{prop:conormal_product}, integrating by parts,
\begin{equation}\label{eq:divu_high_momentum_commutators}
\begin{aligned}
&\left|
\int_0^t\int_\Omega
(D_1^\alpha+D_2^\alpha)\cdot\nabla Z^\alpha\diver u\,dx\,d\tau
\right|
\\
&\leq
\delta
\int_0^t
\|\nabla Z^\alpha\diver u(\tau)\|_{L^2}^2\,d\tau
+
C_\delta C_{m+2}[1+P(Q(t))]
\int_0^t
\mathcal E_m(\tau)\,d\tau .
\end{aligned}
\end{equation}

Substituting \eqref{eq:electric_divu_high_final},
\eqref{eq:divu_high_inertial}, \eqref{eq:divu_high_pressure},
\eqref{eq:divu_high_viscosity}, and
\eqref{eq:divu_high_momentum_commutators} into
\eqref{eq:divu_high_identity}, we
obtain \eqref{eq:divu_high_conormal_estimate}. This completes the proof.
\end{proof}

It remains to handle the pure time derivative case, which is not covered by
Lemma~\ref{lem:divu_high_conormal}. In this case no spatial conormal derivative
is available for tangential integration by parts, and the corresponding
estimate is therefore taken with an additional factor \(\varepsilon\).

\begin{lemma}[Pure time derivative estimate]\label{lem:pure_time_divu}
Let \(m\geq1\). Then it holds that
\begin{equation}\label{eq:pure_time_divu_estimate}
\begin{aligned}
&\sup_{0\leq \tau\leq t}
\varepsilon
\int_\Omega
\left(
\rho|\partial_t^{m-1}\diver u|^2
+
\frac{1}{\gamma p}
|\partial_t^{m-1}\nabla p|^2
\right)(\tau)\,dx  \\
&\quad
+
\varepsilon^2
\int_0^t
\|\nabla\partial_t^{m-1}\diver u(\tau)\|_{L^2}^2\,d\tau  \\
&\leq
C\varepsilon
\int_\Omega
\left(
\rho_0|\partial_t^{m-1}\diver u_0|^2
+
\frac{1}{\gamma p_0}
|\partial_t^{m-1}\nabla p_0|^2
\right)\,dx  \\
&\quad
+
C_{m+1}[1+P(Q(t))]
\int_0^t
\mathcal E_m(\tau)\,d\tau .
\end{aligned}
\end{equation}
\end{lemma}

\begin{proof}
Applying \(\partial_t^{m-1}\) to \eqref{eq:momentum_curl_form} and testing it
against \(\varepsilon\nabla\partial_t^{m-1}\diver u\) on
\(\Omega\times(0,t)\), we obtain
\begin{equation}\label{eq:pure_time_identity}
\begin{aligned}
&\varepsilon
\int_0^t\!\!\int_\Omega
\left(
\rho\partial_t^m u
+
\rho u\cdot\nabla\partial_t^{m-1}u
+
\partial_t^{m-1}\nabla p
\right)
\cdot
\nabla\partial_t^{m-1}\diver u\,dx\,d\tau
\\
={}&
\varepsilon
\int_0^t\!\!\int_\Omega
\partial_t^{m-1}(\rho\nabla\phi)
\cdot
\nabla\partial_t^{m-1}\diver u\,dx\,d\tau
\\
&-
\mu\varepsilon^2
\int_0^t\!\!\int_\Omega
\nabla\times\partial_t^{m-1}\omega
\cdot
\nabla\partial_t^{m-1}\diver u\,dx\,d\tau\\
&
+
(2\mu+\lambda)\varepsilon^2
\int_0^t
\|\nabla\partial_t^{m-1}\diver u\|_{L^2}^2\,d\tau
+
\varepsilon
\int_0^t\!\!\int_\Omega
\mathcal C_t^{m-1}
\cdot
\nabla\partial_t^{m-1}\diver u\,dx\,d\tau .
\end{aligned}
\end{equation}

where
\[
\mathcal C_t^{m-1}
:=
-[\partial_t^{m-1},\rho]u_t
-
[\partial_t^{m-1},\rho u\cdot\nabla]u .
\]

We first estimate the electric force term. Since time derivatives commute with
spatial derivatives,
\[
\partial_t^{m-1}(\rho\nabla\phi)
=
\rho\nabla\partial_t^{m-1}\phi
+
[\partial_t^{m-1},\rho]\nabla\phi .
\]
Hence, by Young's inequality and Proposition~\ref{prop:conormal_product},
\begin{equation}\label{eq:pure_time_electric_final}
\begin{aligned}
&\left|
\varepsilon
\int_0^t\!\!\int_\Omega
\partial_t^{m-1}(\rho\nabla\phi)
\cdot
\nabla\partial_t^{m-1}\diver u\,dx\,d\tau
\right|  \\
&\leq
\varepsilon
\int_0^t
\left(
\|\rho\nabla\partial_t^{m-1}\phi\|_{L^2}
+
\|[\partial_t^{m-1},\rho]\nabla\phi\|_{L^2}
\right)
\|\nabla\partial_t^{m-1}\diver u\|_{L^2}\,d\tau  \\
&\leq
\frac{2\mu+\lambda}{8}
\varepsilon^2
\int_0^t
\|\nabla\partial_t^{m-1}\diver u(\tau)\|_{L^2}^2\,d\tau
+
C_{m+1}[1+P(Q(t))]
\int_0^t
\mathcal E_m(\tau)\,d\tau .
\end{aligned}
\end{equation}

By the same argument as in Lemmas~\ref{lem:div_gradrho_gradphi} and
\ref{lem:divu_high_conormal}, using \eqref{eq:intro_Navier_boundary},
\eqref{eq:pressure_div_relation}, Proposition~\ref{prop:anisotropic_trace},
Proposition~\ref{prop:conormal_product}, and Young's inequality, we estimate
the standard terms as follows.

The inertial term can be estimated by
\begin{equation}\label{eq:pure_time_inertial}
\begin{aligned}
&\varepsilon
\int_0^t\int_\Omega
\left(
\rho\partial_t^{m-1} u_t
+
\rho u\cdot\nabla\partial_t^{m-1}u
\right)
\cdot
\nabla\partial_t^{m-1}\diver u\,dx\,d\tau
\\
={}&
-\varepsilon
\int_0^t\int_\Omega
\rho
\left(
\partial_t^{m-1}\diver u_t
+
u\cdot\nabla\partial_t^{m-1}\diver u
\right)
\partial_t^{m-1}\diver u\,dx\,d\tau
\\
&\quad
-\varepsilon
\int_0^t\int_\Omega
\left(
\nabla\rho\cdot\partial_t^{m-1}u_t
+
\left(\nabla(\rho u)\right)^{T}:
\nabla\partial_t^{m-1}u
\right)
\partial_t^{m-1}\diver u\,dx\,d\tau
\\
&\quad
-\varepsilon
\int_0^t\int_{\partial\Omega}
\rho\,(u\cdot\nabla)n\cdot\partial_t^{m-1}u\,
\partial_t^{m-1}\diver u\,dS\,d\tau
\\
\leq{}&
-\frac{\varepsilon}{2}
\int_\Omega
\rho|\partial_t^{m-1}\diver u|^2(t)\,dx
+
\frac{\varepsilon}{2}
\int_\Omega
\rho_0|\partial_t^{m-1}\diver u_0|^2\,dx
\\
&\quad
+
\frac{2\mu+\lambda}{16}
\varepsilon^2
\int_0^t
\|\nabla\partial_t^{m-1}\diver u\|_{L^2}^2\,d\tau
+
C_{m+1}[1+P(Q(t))]
\int_0^t
\mathcal E_m(\tau)\,d\tau .
\end{aligned}
\end{equation}

The pressure term can be estimated by
\begin{equation}\label{eq:pure_time_pressure}
\begin{aligned}
&\varepsilon
\int_0^t\int_\Omega
\partial_t^{m-1}\nabla p
\cdot
\nabla\partial_t^{m-1}\diver u\,dx\,d\tau
\\
&\leq
-\frac{\varepsilon}{2}
\int_\Omega
\frac{1}{\gamma p}
|\partial_t^{m-1}\nabla p|^2(t)\,dx
+
\frac{\varepsilon}{2}
\int_\Omega
\frac{1}{\gamma p_0}
|\partial_t^{m-1}\nabla p_0|^2\,dx
\\
&\quad
+
C_{m+1}[1+P(Q(t))]
\int_0^t
\mathcal E_m(\tau)\,d\tau .
\end{aligned}
\end{equation}

For the viscous term, we have
\begin{equation}\label{eq:pure_time_viscosity}
\begin{aligned}
&-\mu\varepsilon^2
\int_0^t\int_\Omega
\nabla\times\partial_t^{m-1}\omega
\cdot
\nabla\partial_t^{m-1}\diver u\,dx\,d\tau
\\
&\leq
\frac{2\mu+\lambda}{16}
\varepsilon^2
\int_0^t
\|\nabla\partial_t^{m-1}\diver u\|_{L^2}^2\,d\tau
+
C_{m+1}[1+P(Q(t))]
\int_0^t
\mathcal E_m(\tau)\,d\tau .
\end{aligned}
\end{equation}

The commutator term is controlled by
\begin{equation}\label{eq:pure_time_commutator}
\begin{aligned}
&\left|
\varepsilon
\int_0^t\int_\Omega
\mathcal C_t^{m-1}
\cdot
\nabla\partial_t^{m-1}\diver u\,dx\,d\tau
\right|
\\
&\leq
\frac{2\mu+\lambda}{16}
\varepsilon^2
\int_0^t
\|\nabla\partial_t^{m-1}\diver u\|_{L^2}^2\,d\tau
+
C_{m+1}[1+P(Q(t))]
\int_0^t
\mathcal E_m(\tau)\,d\tau .
\end{aligned}
\end{equation}

Substituting \eqref{eq:pure_time_electric_final},
\eqref{eq:pure_time_inertial}, \eqref{eq:pure_time_pressure},
\eqref{eq:pure_time_viscosity}, and \eqref{eq:pure_time_commutator} into
\eqref{eq:pure_time_identity}, we obtain
\eqref{eq:pure_time_divu_estimate}. This completes the proof.
\end{proof}

The preceding estimates provide the conormal control of \((p,u)\) and the
\(\varepsilon\)-weighted estimate for the highest pure time derivative of
\(\diver u\). However, in order to close the normal derivative
estimates, an estimate for
\(\partial_t^{m-1}\diver u\) is still required.  To avoid a circular estimate, we introduce
an auxiliary quantity
\begin{equation}\label{eq:Lambda_m_1_def}
\Lambda_m^1(t)
:=
\|(p,u)(t)\|_{\mathcal H^m}^2
+
\sum_{|\beta|\leq m-2}
\|Z^\beta\nabla p(t)\|_1^2
+
\sum_{|\beta|\leq m-2}
\|Z^\beta\nabla u(t)\|_1^2 .
\end{equation}
This quantity does not contain the term
\(\|\partial_t^{m-1}\diver u(t)\|_{L^2}^2\).

\begin{lemma}\label{lem:pure_time_divu_without_epsilon}
Let \(m\geq 3\). Then
\begin{equation}\label{eq:pure_time_divu_without_epsilon}
\|\partial_t^{m-1}\diver u(t)\|_{L^2}^2
\leq
C_2[1+P(Q(t))]P\left(\Lambda_m^1(t)\right).
\end{equation}
\end{lemma}

\begin{proof}
By \eqref{eq:pressure_div_relation},
\[
\diver u
=
-\frac{p_t}{\gamma p}
-
\frac{u\cdot\nabla p}{\gamma p}.
\]
Applying \(\partial_t^{m-1}\) to this identity gives
\begin{equation}\label{eq:dtm_divu_decomp_m9}
\begin{aligned}
\partial_t^{m-1}\diver u
={}&
-\frac{1}{\gamma p}\partial_t^m p
-
\frac{u}{\gamma p}\cdot\nabla\partial_t^{m-1}p  \\
&-
\left[
\partial_t^{m-1},\frac{1}{\gamma p}
\right]p_t
-
\left[
\partial_t^{m-1},\frac{u}{\gamma p}
\right]\cdot\nabla p .
\end{aligned}
\end{equation}
The two leading terms are controlled by the definition of
\(\Lambda_m^1(t)\):
\begin{equation}\label{eq:dtm_divu_leading_est}
\left\|
\frac{1}{\gamma p}\partial_t^m p
\right\|_{L^2}^2
+
\left\|
\frac{u}{\gamma p}\cdot\nabla\partial_t^{m-1}p
\right\|_{L^2}^2
\leq
C[1+P(Q(t))]\Lambda_m^1(t).
\end{equation}
For the commutator terms, Proposition~\ref{prop:conormal_product} yields
\begin{equation}\label{eq:dtm_divu_commutator_est}
\begin{aligned}
&\left\|
\left[
\partial_t^{m-1},\frac{1}{\gamma p}
\right]p_t
\right\|_{L^2}^2
+
\left\|
\left[
\partial_t^{m-1},\frac{u}{\gamma p}
\right]\cdot\nabla p
\right\|_{L^2}^2  \\
&\leq
C_2[1+P(Q(t))]P\left(\Lambda_m^1(t)\right).
\end{aligned}
\end{equation}
Combining \eqref{eq:dtm_divu_decomp_m9}--
\eqref{eq:dtm_divu_commutator_est}, we obtain
\eqref{eq:pure_time_divu_without_epsilon}. This completes the proof.
\end{proof}
\begin{lemma}
\label{lem:grad_high_time_p}
Let \(m\geq1\). Then
\begin{equation}\label{eq:grad_high_time_p_estimate}
\begin{aligned}
\int_0^t
\|\nabla\partial_t^{m-1}p(\tau)\|_{L^2}^2\,d\tau
\leq{}&
C_{m+2}[1+P(Q(t))]
\int_0^t
\mathcal E_m(\tau)\,d\tau  \\
&+
C\varepsilon^2
\int_0^t
\|\nabla^2u(\tau)\|_{\mathcal H^{m-1}}^2\,d\tau .
\end{aligned}
\end{equation}
\end{lemma}

\begin{proof}
Applying \(\partial_t^{m-1}\) to \eqref{eq:momentum_curl_form} and solving for
\(\nabla\partial_t^{m-1}p\), we get
\begin{equation}\label{eq:grad_high_time_p_identity}
\begin{aligned}
\nabla\partial_t^{m-1}p
={}&
-\partial_t^{m-1}(\rho u_t)
-
\partial_t^{m-1}(\rho u\cdot\nabla u)
-\mu\varepsilon\partial_t^{m-1}\nabla\times\omega  \\
&+
(2\mu+\lambda)\varepsilon\partial_t^{m-1}\nabla\diver u
+
\partial_t^{m-1}(\rho\nabla\phi).
\end{aligned}
\end{equation}
We estimate each term on the right-hand side of
\eqref{eq:grad_high_time_p_identity}. First,
\[
\partial_t^{m-1}(\rho u_t)
=
\rho\partial_t^m u
+
[\partial_t^{m-1},\rho]u_t .
\]
Thus, by Proposition~\ref{prop:conormal_product},
\begin{equation}\label{eq:gradp_time_rho_ut}
\int_0^t
\|\partial_t^{m-1}(\rho u_t)(\tau)\|_{L^2}^2\,d\tau
\leq
C_{m+1}[1+P(Q(t))]
\int_0^t
\mathcal E_m(\tau)\,d\tau .
\end{equation}
Similarly,
\[
\partial_t^{m-1}(\rho u\cdot\nabla u)
=
\rho u\cdot\nabla\partial_t^{m-1}u
+
[\partial_t^{m-1},\rho u]\cdot\nabla u ,
\]
and Proposition~\ref{prop:conormal_product} gives
\begin{equation}\label{eq:gradp_time_transport}
\int_0^t
\|\partial_t^{m-1}(\rho u\cdot\nabla u)(\tau)\|_{L^2}^2\,d\tau
\leq
C_{m+1}[1+P(Q(t))]
\int_0^t
\mathcal E_m(\tau)\,d\tau .
\end{equation}
For the viscous terms, we directly have
\begin{equation}\label{eq:gradp_time_viscous}
\begin{aligned}
&\varepsilon^2
\int_0^t
\|\partial_t^{m-1}(\nabla\times\omega,\nabla\diver u)(\tau)\|_{L^2}^2\,d\tau
\leq
C\varepsilon^2
\int_0^t
\|\nabla^2u(\tau)\|_{\mathcal H^{m-1}}^2\,d\tau .
\end{aligned}
\end{equation}

It remains to estimate the electric force term. Since
\[
\begin{aligned}
\partial_t^{m-1}(\rho\nabla\phi)
&=
\rho\nabla\partial_t^{m-1}\phi
+
[\partial_t^{m-1},\rho]\nabla\phi, \\
[\partial_t^{m-1},\rho]\nabla\phi
&=
\sum_{j=1}^{m-1}
C_{m-1}^j
\partial_t^j\rho\,
\nabla\partial_t^{m-1-j}\phi .
\end{aligned}
\]

where \(C_{m-1}^j\) denotes the usual binomial coefficient.
Proposition~\ref{prop:conormal_product} gives
\begin{equation}\label{eq:gradp_time_electric_estimate}
\begin{aligned}
\int_0^t
\|\partial_t^{m-1}(\rho\nabla\phi)(\tau)\|_{L^2}^2\,d\tau
&\leq
C_{m+1}[1+P(Q(t))]
\int_0^t
\mathcal E_m(\tau)\,d\tau .
\end{aligned}
\end{equation}

Combining \eqref{eq:grad_high_time_p_identity}--
\eqref{eq:gradp_time_electric_estimate}, we obtain
\eqref{eq:grad_high_time_p_estimate}. This completes the proof.
\end{proof}

The estimates obtained above still contain the term
\(\|\nabla Z^{m-2}\diver u\|_{L^2}^2\). A direct estimate of the corresponding
second-order spatial derivatives is not convenient in a general domain. Instead,
we use the pressure equation to rewrite \(\diver u\), so that the tangential and
time derivative parts are controlled by the conormal energy, while the normal
part is treated by the Hardy inequality near the boundary.

\begin{lemma}
\label{lem:normal_divu_estimate}
Let \(m\geq2\). Then
\begin{equation}\label{eq:normal_divu_estimate}
\int_0^t
\|\nabla Z^{m-2}\diver u(\tau)\|_{L^2}^2\,d\tau
\leq
C\int_0^t\|\nabla\partial_t^{m-1}p\|_{L^2}^2\,d\tau
+C_m[1+P(Q(t))]
\int_0^t
\mathcal E_m(\tau)\,d\tau .
\end{equation}
\end{lemma}

\begin{proof}
By \eqref{eq:pressure_div_relation}, we may write
\[
\diver u
=
-\frac{1}{\gamma}
\left(
\partial_t\ln p
+
u\cdot\nabla\ln p
\right).
\]
Let \(|\alpha|\leq m-2\). In a boundary coordinate patch, we decompose
\[
u\cdot\nabla\ln p
=
u_y\cdot\nabla_y\ln p
+
(u\cdot N)\partial_z\ln p .
\]
Hence
\begin{equation}\label{eq:divu_normal_decomposition}
\begin{aligned}
\int_0^t
\|\nabla Z^\alpha\diver u\|_{L^2}^2\,d\tau
\leq{}&
C
\int_0^t
\Big(
\|\nabla Z^\alpha\partial_t\ln p\|_{L^2}^2
+
\|\nabla Z^\alpha(u_y\cdot\nabla_y\ln p)\|_{L^2}^2  \\
&\qquad\qquad
+
\|\nabla Z^\alpha((u\cdot N)\partial_z\ln p)\|_{L^2}^2
\Big)\,d\tau .
\end{aligned}
\end{equation}

The first two terms are estimated as in
\cite[Lemma~3.12, (3.98)--(3.100)]{WangXinYong2015}.
The case \(Z^\alpha=\partial_t^{m-2}\) contains
\(\nabla\partial_t^{m-1}\ln p\), so its highest pressure derivative must
be retained without an \(\varepsilon\) weight. The other mixed derivatives
are contained in \(\mathcal E_m\). Thus Proposition~\ref{prop:conormal_product}
gives
\begin{equation}\label{eq:divu_time_tangential_terms}
\begin{aligned}
&\int_0^t
\|\nabla Z^\alpha\partial_t\ln p\|_{L^2}^2\,d\tau
+
\int_0^t
\|\nabla Z^\alpha(u_y\cdot\nabla_y\ln p)\|_{L^2}^2\,d\tau  \\
&\leq
C\int_0^t\|\nabla\partial_t^{m-1}p\|_{L^2}^2\,d\tau
+C_m[1+P(Q(t))]
\int_0^t
\mathcal E_m(\tau)\,d\tau .
\end{aligned}
\end{equation}

It remains to estimate the normal part. Since \(u\cdot n=0\) on
\(\partial\Omega\), the coefficient \(u\cdot N\) vanishes on the boundary. Using
\(Z_3=\varphi(z)\partial_z\) with \(\varphi(z)\sim z\) near
\(\partial\Omega\), Proposition~\ref{prop:conormal_product}, and the Hardy
inequality, we obtain
\begin{equation}\label{eq:divu_normal_part_estimate}
\begin{aligned}
&\int_0^t
\|\nabla Z^\alpha((u\cdot N)\partial_z\ln p)(\tau)\|_{L^2}^2\,d\tau  \\
&\leq
\int_0^t
\|\nabla((u\cdot N)\partial_z\ln p)(\tau)\|_{\mathcal H^{m-2}}^2\,d\tau  \\
&\leq
C
\int_0^t
\|\nabla(u\cdot N)\,\partial_z\ln p\|_{\mathcal H^{m-2}}^2\,d\tau
+
C
\int_0^t
\left\|
\frac{u\cdot N}{\varphi(z)}\,Z_3\nabla\ln p
\right\|_{\mathcal H^{m-2}}^2\,d\tau  \\
&\leq
C_m[1+P(Q(t))]
\int_0^t
\left(
\mathcal E_m(\tau)
+
\sum_{|\beta|\leq1}
\left\|
Z^\beta\left(\frac{u\cdot N}{\varphi(z)}\right)(\tau)
\right\|_{\mathcal H^{m-2}}^2
\right)d\tau  \\
&\leq
C_m[1+P(Q(t))]
\int_0^t
\mathcal E_m(\tau)\,d\tau .
\end{aligned}
\end{equation}
Here \(N\) denotes the normal direction in the boundary coordinate patch,
\(z\) is the normal coordinate, and \(\varphi(z)\sim z\) is the coefficient in
the conormal vector field \(Z_3=\varphi(z)\partial_z\). In the last step we used
the Hardy inequality
\[
\sum_{|\beta|\leq1}
\left\|
Z^\beta\left(\frac{u\cdot N}{\varphi(z)}\right)
\right\|_{\mathcal H^{m-2}}
\leq
C_m
\left(
\|\nabla u\|_{\mathcal H^{m-1}}
+
\|u\|_{\mathcal H^{m-1}}
\right),
\]
which follows from \(u\cdot n=0\) on \(\partial\Omega\).

Combining \eqref{eq:divu_normal_decomposition}--
\eqref{eq:divu_normal_part_estimate}, summing over \(|\alpha|\leq m-2\), and
using a partition of unity, we obtain \eqref{eq:normal_divu_estimate}. This
completes the proof.
\end{proof}

Combining \eqref{eq:normal_divu_estimate} with
\eqref{eq:grad_high_time_p_estimate}, and using
\(\varepsilon\|\nabla\partial_t^{m-1}p\|_{L^2}^2\leq\mathcal E_m\),
we obtain, for \(0<\delta\leq1\),
\begin{equation}\label{eq:normal_divu_weighted_closure}
\begin{aligned}
&(\delta+\varepsilon)\int_0^t
\|\nabla Z^{m-2}\diver u\|_{L^2}^2\,d\tau
+\delta\int_0^t\|\nabla\partial_t^{m-1}p\|_{L^2}^2\,d\tau\\
&\leq C\delta\varepsilon^2\int_0^t
\|\nabla^2u\|_{\mathcal H^{m-1}}^2\,d\tau
+C_m[1+P(Q(t))]\int_0^t\mathcal E_m(\tau)\,d\tau.
\end{aligned}
\end{equation}
Here \(C\) is independent of \(\delta\) and \(\varepsilon\).

\subsection{Normal derivative estimates}

In order to estimate \(\|\nabla u\|_{\mathcal H^{m-1}}\), it remains to
control the normal derivative near the boundary. Let \(\chi\) be a cut-off
function compactly supported in one boundary coordinate patch and equal to one
in a neighborhood of the boundary. The tangential derivatives satisfy
\[
\|\chi\partial_{y_i}u\|_{\mathcal H^{m-1}}
\leq
C\|u\|_{\mathcal H^m},
\qquad i=1,2.
\]
Thus it suffices to estimate \(\|\chi\partial_nu\|_{\mathcal H^{m-1}}\).

In the local coordinates, one has
\begin{equation}\label{eq:normal_div_decomposition}
\diver u
=
\partial_nu\cdot n
+
(\Pi\partial_{y_1}u)_1
+
(\Pi\partial_{y_2}u)_2,
\end{equation}
and
\begin{equation}\label{eq:normal_projection_decomposition}
\partial_nu
=
(\partial_nu\cdot n)n
+
\Pi(\partial_nu).
\end{equation}
It follows from \eqref{eq:normal_div_decomposition} and
\eqref{eq:normal_projection_decomposition} that
\begin{equation}\label{eq:normal_derivative_reduction}
\begin{aligned}
\|\chi\partial_nu\|_{\mathcal H^{m-1}}
&\leq
\|\chi\partial_nu\cdot n\|_{\mathcal H^{m-1}}
+
\|\chi\Pi(\partial_nu)\|_{\mathcal H^{m-1}}  \\
&\leq
C_m
\left(
\|\chi\diver u\|_{\mathcal H^{m-1}}
+
\|\chi\Pi(\partial_nu)\|_{\mathcal H^{m-1}}
+
\|u\|_{\mathcal H^m}
\right).
\end{aligned}
\end{equation}
The terms \(\|u\|_{\mathcal H^m}\) and
\(\|\chi\diver u\|_{\mathcal H^{m-1}}\) have already been controlled in the
previous estimates. Hence it is enough to estimate
\(\|\chi\Pi(\partial_nu)\|_{\mathcal H^{m-1}}\).

We extend the boundary matrix \(B\) into the local coordinate patch by setting
\(B(y,z)=B(y)\). Define
\begin{equation}\label{eq:eta_def}
\eta
:=
\chi\left(
\omega\times n+\Pi(Bu)
\right)
=
\chi\left(
\Pi(\omega\times n)+\Pi(Bu)
\right).
\end{equation}
By the Navier--slip boundary condition,
\[
n\times\omega=[Bu]_\tau
\quad\text{on }\partial\Omega,
\]
we have
\begin{equation}\label{eq:eta_boundary}
\eta|_{\partial\Omega}=0.
\end{equation}
Moreover, since
\[
\omega\times n
=
(\nabla u-(\nabla u)^T)n,
\]
we can rewrite \(\eta\) as
\begin{equation}\label{eq:eta_rewrite}
\eta
=
\chi\left[
\Pi(\partial_nu)
-
\Pi(\nabla(u\cdot n))
+
\Pi((\nabla n)^Tu)
+
\Pi(Bu)
\right].
\end{equation}
Therefore,
\begin{equation}\label{eq:normal_by_eta}
\|\chi\Pi(\partial_nu)\|_{\mathcal H^{m-1}}
\leq
C_{m+1}
\left(
\|\eta\|_{\mathcal H^{m-1}}
+
\|u\|_{\mathcal H^m}
\right).
\end{equation}
It remains to estimate \(\eta\).

\begin{lemma}\label{lem:eta_estimate}
Let \(m\geq 1\). Then it holds that
\begin{equation}\label{eq:eta_estimate}
\begin{aligned}
&\sup_{0\leq \tau\leq t}
\|\eta(\tau)\|_{\mathcal H^{m-1}}^2
+
\varepsilon
\int_0^t
\|\nabla\eta(\tau)\|_{\mathcal H^{m-1}}^2\,d\tau  \\
&\leq
C_{m+2}
\Bigg\{
\|\eta(0)\|_{\mathcal H^{m-1}}^2
+
\delta\varepsilon^2
\int_0^t
\|\nabla^2u(\tau)\|_{\mathcal H^{m-1}}^2\,d\tau  \\
&\qquad
+
C_\delta[1+P(Q(t))]
\int_0^t
P(\mathcal E_m(\tau))\,d\tau
\Bigg\}.
\end{aligned}
\end{equation}
\end{lemma}

\begin{proof}
Set
\[
\mathcal T f:=\chi\Pi(f\times n),
\qquad
\mathcal B f:=\chi\Pi(Bf),
\qquad
\mathcal L:=\rho\partial_t+\rho u\cdot\nabla-\mu\varepsilon\Delta .
\]
Then
\[
\eta=\mathcal T\omega+\mathcal B u .
\]
Taking curl to \eqref{eq:momentum_curl_form}, one obtains
\begin{equation}\label{eq:vorticity_equation_eta}
\rho\omega_t+\rho u\cdot\nabla\omega-\mu\varepsilon\Delta\omega
=
F_\omega^0+\nabla\rho\times\nabla\phi,
\end{equation}
where
\[
F_\omega^0
:=
-\nabla\rho\times u_t
-\nabla\rho\times(u\cdot\nabla u)
+\rho(\omega\cdot\nabla)u
-\rho(\diver u)\omega .
\]
Using \eqref{eq:vorticity_equation_eta} and \eqref{eq:momentum_curl_form}, we
get
\begin{equation}\label{eq:eta_equation}
\begin{aligned}
\rho\eta_t+\rho u\cdot\nabla\eta-\mu\varepsilon\Delta\eta
&=
F_\eta^0+F_\eta^E,
\qquad
\eta|_{\partial\Omega}=0,
\end{aligned}
\end{equation}
where
\begin{equation}\label{eq:eta_sources}
\begin{aligned}
F_\eta^0
&:=
\mathcal T F_\omega^0
+
\mathcal B\left(
-\nabla p+(\mu+\lambda)\varepsilon\nabla\diver u
\right)
+
\mathcal R_\eta,  \\
F_\eta^E
&:=
\mathcal T(\nabla\rho\times\nabla\phi)
+
\mathcal B(\rho\nabla\phi),
\end{aligned}
\end{equation}
and
\[
\mathcal R_\eta
:=
[\mathcal L,\mathcal T]\omega
+
[\mathcal L,\mathcal B]u .
\]
Here
\[
[\mathcal L,\mathcal T]\omega
=
\mathcal L(\mathcal T\omega)-\mathcal T(\mathcal L\omega),
\qquad
[\mathcal L,\mathcal B]u
=
\mathcal L(\mathcal Bu)-\mathcal B(\mathcal Lu).
\]

Applying \(Z^\alpha\), \(|\alpha|\leq m-1\), to
\eqref{eq:eta_equation}, we obtain
\begin{equation}\label{eq:eta_conormal_equation}
\begin{aligned}
\rho\partial_t Z^\alpha\eta
+
\rho u\cdot\nabla Z^\alpha\eta
-
\mu\varepsilon\Delta Z^\alpha\eta
=
Z^\alpha F_\eta^0
+
Z^\alpha F_\eta^E
+
\mathcal C_\eta^\alpha ,
\end{aligned}
\end{equation}
where
\begin{equation}\label{eq:eta_commutator}
\mathcal C_\eta^\alpha
:=
-[Z^\alpha,\rho]\eta_t
-
[Z^\alpha,\rho u\cdot\nabla]\eta
+
\mu\varepsilon[Z^\alpha,\Delta]\eta .
\end{equation}
Multiplying \eqref{eq:eta_conormal_equation} by \(Z^\alpha\eta\) and
integrating over \(\Omega\), using \(\eta|_{\partial\Omega}=0\), gives
\begin{equation}\label{eq:eta_energy_identity}
\begin{aligned}
&\frac12\frac{d}{dt}
\int_\Omega
\rho|Z^\alpha\eta|^2\,dx
+
\mu\varepsilon
\|\nabla Z^\alpha\eta\|_{L^2}^2  \\
&=
\int_\Omega
Z^\alpha F_\eta^0\cdot Z^\alpha\eta\,dx
+
\int_\Omega
Z^\alpha F_\eta^E\cdot Z^\alpha\eta\,dx
+
\int_\Omega
\mathcal C_\eta^\alpha\cdot Z^\alpha\eta\,dx .
\end{aligned}
\end{equation}

By Proposition~\ref{prop:conormal_product},
Proposition~\ref{prop:anisotropic_trace}, and Young's inequality, together
with the estimates obtained above,
\begin{equation}\label{eq:eta_standard_source_estimate}
\begin{aligned}
&\sum_{|\alpha|\leq m-1}
\int_0^t
\left|
\int_\Omega
\left(
Z^\alpha F_\eta^0+\mathcal C_\eta^\alpha
\right)\cdot Z^\alpha\eta\,dx
\right|d\tau  \\
&\leq
\delta\varepsilon
\int_0^t
\|\nabla\eta(\tau)\|_{\mathcal H^{m-1}}^2\,d\tau
+
\delta\varepsilon^2
\int_0^t
\|\nabla^2u(\tau)\|_{\mathcal H^{m-1}}^2\,d\tau  \\
&\quad
+
C_\delta C_{m+2}[1+P(Q(t))]
\int_0^t
P(\mathcal E_m(\tau))\,d\tau .
\end{aligned}
\end{equation}

It remains to estimate the electric force term. The highest pure time
derivative of \(\nabla\rho\) must be retained, just as in
Lemma~\ref{lem:normal_divu_estimate}. The time-integrated product estimate
and the Poisson estimate give
\begin{equation}\label{eq:eta_electric_source_bound}
\begin{aligned}
\int_0^t\|F_\eta^E\|_{\mathcal H^{m-1}}^2\,d\tau
\leq{}&C_{m+2}[1+P(Q(t))]\int_0^t
\left(\|\nabla\partial_t^{m-1}p\|_{L^2}^2
+P(\mathcal E_m(\tau))\right)\,d\tau.
\end{aligned}
\end{equation}
More precisely, the only derivatives not controlled by \(\mathcal E_m\)
occur in
\(\mathcal T(\nabla\partial_t^{m-1}\rho\times\nabla\phi)\).
Since \(p=\rho^\gamma\), their leading part is a bounded coefficient
times \(\nabla\partial_t^{m-1}p\). Also
\(\|\eta\|_{\mathcal H^{m-1}}^2\leq C\mathcal E_m\).
Estimating this leading part by Young's inequality before using
\eqref{eq:grad_high_time_p_estimate}, we obtain
\begin{equation}\label{eq:eta_electric_energy_estimate}
\begin{aligned}
&\sum_{|\alpha|\leq m-1}\int_0^t
\left|\int_\Omega Z^\alpha F_\eta^E\cdot Z^\alpha\eta\,dx\right|d\tau\\
&\leq\delta\int_0^t\|\nabla\partial_t^{m-1}p\|_{L^2}^2\,d\tau+C_\delta C_{m+2}[1+P(Q(t))]
\int_0^tP(\mathcal E_m(\tau))\,d\tau\\
&\leq C\delta\varepsilon^2\int_0^t
\|\nabla^2u\|_{\mathcal H^{m-1}}^2\,d\tau
+C_\delta C_{m+2}[1+P(Q(t))]
\int_0^tP(\mathcal E_m(\tau))\,d\tau .
\end{aligned}
\end{equation}

Summing \eqref{eq:eta_energy_identity} over \(|\alpha|\leq m-1\), integrating
over \([0,t]\), and using \eqref{eq:eta_standard_source_estimate} and
\eqref{eq:eta_electric_energy_estimate}, we obtain
\[
\begin{aligned}
&\|\eta(t)\|_{\mathcal H^{m-1}}^2
+
\varepsilon
\int_0^t
\|\nabla\eta(\tau)\|_{\mathcal H^{m-1}}^2\,d\tau  \\
&\leq
C_{m+2}
\left(
\|\eta(0)\|_{\mathcal H^{m-1}}^2
\right)
+
\delta\varepsilon^2
\int_0^t
\|\nabla^2u(\tau)\|_{\mathcal H^{m-1}}^2\,d\tau  \\
&\quad
+
C_\delta C_{m+2}[1+P(Q(t))]
\int_0^t
P(\mathcal E_m(\tau))\,d\tau
+
C
\int_0^t
\|\eta(\tau)\|_{\mathcal H^{m-1}}^2\,d\tau .
\end{aligned}
\]
Taking \(\delta>0\) sufficiently small and applying Gronwall's inequality gives
\eqref{eq:eta_estimate}. This completes the proof.
\end{proof}

It follows from \eqref{eq:normal_div_decomposition}--\eqref{eq:normal_derivative_reduction}
that
\begin{equation}\label{eq:Zbeta_grad_u_H1_control}
\sum_{|\beta|\leq m-2}
\|Z^\beta\nabla u\|_1^2
\leq
C_{m+1}
\left(
\|u\|_{\mathcal H^m}^2
+
\|\eta\|_{\mathcal H^{m-1}}^2
+
\sum_{k=0}^{m-2}
\|\partial_t^k\diver u\|_{m-1-k}^2
\right).
\end{equation}
Moreover,
\begin{equation}\label{eq:nabla2_u_Hm1_control}
\begin{aligned}
\int_0^t
\|\nabla^2u(\tau)\|_{\mathcal H^{m-1}}^2\,d\tau
\leq{}&
C_{m+2}
\int_0^t
\Big(
\|\nabla u(\tau)\|_{\mathcal H^m}^2
+
\|\nabla\eta(\tau)\|_{\mathcal H^{m-1}}^2  \\
&\qquad
+
\|\nabla\diver u(\tau)\|_{\mathcal H^{m-1}}^2
+
\mathcal E_m(\tau)
\Big)\,d\tau .
\end{aligned}
\end{equation}

Similarly,
\begin{equation}\label{eq:nabla2_dtku_control}
\begin{aligned}
&\sum_{k=0}^{m-2}
\int_0^t
\|\nabla^2\partial_t^k u(\tau)\|_{m-1-k}^2\,d\tau  \\
&\leq
C_{m+2}
\Bigg\{
\int_0^t
\left(
\|\nabla u(\tau)\|_{\mathcal H^m}^2
+
\|\nabla\eta(\tau)\|_{\mathcal H^{m-1}}^2
\right)d\tau  \\
&\qquad\qquad
+
\sum_{k=0}^{m-2}
\int_0^t
\|\partial_t^k\nabla\diver u(\tau)\|_{m-1-k}^2\,d\tau
+
\int_0^t
\mathcal E_m(\tau)\,d\tau
\Bigg\}.
\end{aligned}
\end{equation}
In addition,
\begin{equation}\label{eq:epsilon_nabla2_Zm2_u_control}
\varepsilon
\int_0^t
\|\nabla^2 Z^{m-2}u(\tau)\|_{L^2}^2\,d\tau
\leq
C_{m+1}\varepsilon
\int_0^t
\|\nabla\eta(\tau)\|_{\mathcal H^{m-2}}^2\,d\tau
+
C_{m+1}
\int_0^t
\mathcal E_m(\tau)\,d\tau .
\end{equation}

Taking \(\delta>0\) sufficiently small and using
\eqref{eq:high_order_conormal_estimate},
\eqref{eq:divu_high_conormal_estimate},
\eqref{eq:pure_time_divu_estimate},
\eqref{eq:pure_time_divu_without_epsilon},
\eqref{eq:grad_high_time_p_estimate},
\eqref{eq:normal_divu_weighted_closure},
\eqref{eq:eta_estimate}, and
\eqref{eq:Zbeta_grad_u_H1_control}--\eqref{eq:epsilon_nabla2_Zm2_u_control},
we obtain
\begin{equation}\label{eq:L2_closing_estimate}
\begin{aligned}
&\sup_{0\leq \tau\leq t}
\Bigg\{
\mathcal E_m(\tau)
+
\|\eta(\tau)\|_{\mathcal H^{m-1}}^2
+
\varepsilon
\left\|
\left(
\partial_t^{m-1}\diver u,
\partial_t^{m-1}\nabla p
\right)(\tau)
\right\|_{L^2}^2
\Bigg\}  \\
&\quad
+
\varepsilon
\int_0^t
\|\nabla\eta(\tau)\|_{\mathcal H^{m-1}}^2\,d\tau
+
\varepsilon
\int_0^t
\|\nabla u(\tau)\|_{\mathcal H^m}^2\,d\tau  \\
&\quad
+
\varepsilon
\sum_{k=0}^{m-2}
\int_0^t
\|\partial_t^k\nabla\diver u(\tau)\|_{m-1-k}^2\,d\tau
+
\varepsilon^2
\int_0^t
\|\partial_t^{m-1}\nabla\diver u(\tau)\|_{L^2}^2\,d\tau  \\
&\quad
+
\varepsilon
\sum_{k=0}^{m-2}
\int_0^t
\|\nabla^2\partial_t^k u(\tau)\|_{m-1-k}^2\,d\tau
+
\varepsilon^2
\int_0^t
\|\nabla^2u(\tau)\|_{\mathcal H^{m-1}}^2\,d\tau  \\
&\quad
+
\int_0^t
\|\nabla\partial_t^{m-1}p(\tau)\|_{L^2}^2\,d\tau  \\
&\leq
C_{m+2}
\left\{
\mathcal E_m(0)
+
[1+P(Q(t))]
\int_0^t
P(\mathcal E_m(\tau))\,d\tau
\right\}.
\end{aligned}
\end{equation}
\subsection{\texorpdfstring{\(L^\infty\)-estimates}{L-infinity estimates}}

\begin{lemma}\label{lem:Linfty_basic}
Let \(m\geq6\). Then the following estimates hold:
\begin{equation}\label{eq:Linfty_lnrho_p_u}
\|Z^\alpha(\ln\rho,p,u)(t)\|_{L^\infty}^2
\leq
C P\left(\Lambda_m^1(t)\right),
\qquad
m\geq 2+|\alpha|.
\end{equation}
Moreover,
\begin{equation}\label{eq:Linfty_grad_lnrho_p}
\|\nabla(\ln\rho,p)(t)\|_{\mathcal H^{1,\infty}}^2
\leq
C_3
\left[
P\left(\Lambda_m^1(t)\right)
+
P\left(\|\Delta p(t)\|_{\mathcal H^1}^2\right)
\right],
\qquad
m\geq5,
\end{equation}
and
\begin{equation}\label{eq:Linfty_grad_phi}
\|\nabla\phi(t)\|_{\mathcal H^{1,\infty}}^2
\leq
C_3 P\left(\mathcal E_m(t)\right),
\qquad
m\geq5.
\end{equation}
Furthermore,
\begin{equation}\label{eq:Linfty_divu}
\|\diver u(t)\|_{\mathcal H^{1,\infty}}^2
\leq
C_3
\left[
P\left(\Lambda_m^1(t)\right)
+
P\left(\|\Delta p(t)\|_{\mathcal H^1}^2\right)
\right],
\qquad
m\geq5,
\end{equation}
\begin{equation}\label{eq:Linfty_grad_divu_Linfty}
\|\nabla\diver u(t)\|_{L^\infty}^2
\leq
C_3 P(Q(t)),
\end{equation}
and, for any sufficiently small \(\delta>0\),
\begin{equation}\label{eq:Linfty_grad_divu_H1infty}
\begin{aligned}
\|\nabla\diver u(t)\|_{\mathcal H^{1,\infty}}^2
\leq
C_4[1+P(Q(t))]
\bigg[
&C_\delta P\left(\Lambda_m^1(t)\right)
+
C_\delta P\left(\|\Delta p(t)\|_{\mathcal H^1}^2\right)  \\
&+
\delta\|\Delta p(t)\|_{\mathcal H^2}^2
\bigg],
\qquad
m\geq6 .
\end{aligned}
\end{equation}
Consequently,
\begin{equation}\label{eq:Q_reduction}
\begin{aligned}
Q(t)
\leq
C_4
\sup_{0\leq \tau\leq t}
\Big\{
&P\left(\Lambda_m^1(\tau)\right)
+
P\left(\mathcal E_m(\tau)\right)
+
P\left(\|\Delta p(\tau)\|_{\mathcal H^1}^2\right)  \\
&+
\|\nabla u(\tau)\|_{\mathcal H^{1,\infty}}^2
+
\delta\|\Delta p(\tau)\|_{\mathcal H^2}^2
\Big\}.
\end{aligned}
\end{equation}
\end{lemma}

\begin{proof}
We first estimate the lower-order \(L^\infty\)-norms. By
Proposition~\ref{prop:anisotropic_trace}, for \(m\geq2+|\alpha|\),
\begin{equation}\label{eq:Linfty_basic_step}
\begin{aligned}
\|Z^\alpha(\ln\rho,p,u)\|_{L^\infty}^2
&\leq
C
\left(
\|\nabla Z^\alpha(\ln\rho,p,u)\|_{\mathcal H^1}
+
\|Z^\alpha(\ln\rho,p,u)\|_{\mathcal H^1}
\right)  \\
&\qquad\qquad
\times
\|Z^\alpha(\ln\rho,p,u)\|_{\mathcal H^2}
\leq
C P\left(\Lambda_m^1(t)\right).
\end{aligned}
\end{equation}
where \(p=\rho^\gamma\) and Proposition~\ref{prop:conormal_product} have been
used. This proves \eqref{eq:Linfty_lnrho_p_u}.

Next, applying the same anisotropic Sobolev estimate to
\(Z^\alpha\nabla p\), \(|\alpha|\leq1\), gives
\begin{equation}\label{eq:Linfty_gradp_framework}
\begin{aligned}
\|Z^\alpha\nabla p\|_{L^\infty}^2
&\leq
C
\left(
\|\nabla Z^\alpha\nabla p\|_{\mathcal H^1}
+
\|Z^\alpha\nabla p\|_{\mathcal H^1}
\right)
\|Z^\alpha\nabla p\|_{\mathcal H^2}  \\
&\leq
C_3
\left[
P\left(\Lambda_m^1(t)\right)
+
P\left(\|\Delta p(t)\|_{\mathcal H^1}^2\right)
\right].
\end{aligned}
\end{equation}
Here the non-conormal second-order derivatives of \(p\) are recovered from
\(\Delta p\) and tangential derivatives in local coordinates, namely
\[
\|\nabla^2p\|_{\mathcal H^1}
\leq
C_3
\left(
\|\Delta p\|_{\mathcal H^1}
+
\|\nabla p\|_{\mathcal H^1}
+
\|p\|_{\mathcal H^3}
\right).
\]
Together with \(p=\rho^\gamma\), this yields
\eqref{eq:Linfty_grad_lnrho_p}.

For the electric potential, Proposition~\ref{prop:anisotropic_trace} and Proposition~\ref{prop:poisson_neumann_estimate} imply
\begin{equation}\label{eq:Linfty_grad_phi_proof}
\begin{aligned}
\|Z^\alpha\nabla\phi\|_{L^\infty}^2
&\leq
C
\left(
\|\nabla Z^\alpha\nabla\phi\|_{\mathcal H^1}
+
\|Z^\alpha\nabla\phi\|_{\mathcal H^1}
\right)
\|Z^\alpha\nabla\phi\|_{\mathcal H^2}  \\
&\leq
C_3 P\left(\mathcal E_m(t)\right),
\qquad |\alpha|\leq1,
\end{aligned}
\end{equation}
which gives \eqref{eq:Linfty_grad_phi}.

We now estimate \(\diver u\). From \eqref{eq:pressure_div_relation},
\begin{equation}\label{eq:divu_lnrho_relation_Linfty}
\diver u
=
-(\ln\rho)_t
-
u\cdot\nabla\ln\rho .
\end{equation}
Using Proposition~\ref{prop:conormal_product},
\eqref{eq:Linfty_lnrho_p_u}, and \eqref{eq:Linfty_grad_lnrho_p}, we obtain
\begin{equation}\label{eq:Linfty_divu_proof}
\begin{aligned}
\|\diver u(t)\|_{\mathcal H^{1,\infty}}^2
&\leq
C
\left(
\|(\ln\rho)_t(t)\|_{\mathcal H^{1,\infty}}^2
+
\|u(t)\cdot\nabla\ln\rho(t)\|_{\mathcal H^{1,\infty}}^2
\right)  \\
&\leq
C_3
\left[
P\left(\Lambda_m^1(t)\right)
+
P\left(\|\Delta p(t)\|_{\mathcal H^1}^2\right)
\right],
\end{aligned}
\end{equation}
and hence \eqref{eq:Linfty_divu}. Differentiating
\eqref{eq:divu_lnrho_relation_Linfty} once and using the preceding estimates
gives
\[
\|\nabla\diver u(t)\|_{L^\infty}^2
\leq
C_3P(Q(t)),
\]
which is \eqref{eq:Linfty_grad_divu_Linfty}.

It remains to estimate \(\nabla\diver u\) in
\(\mathcal H^{1,\infty}\). Differentiating
\eqref{eq:divu_lnrho_relation_Linfty}, we have
\[
\nabla\diver u
=
-\nabla(\ln\rho)_t
-
(\nabla u)\nabla\ln\rho
-
u\cdot\nabla^2\ln\rho .
\]
Applying \(Z^\alpha\), \(|\alpha|\leq1\), and using
Proposition~\ref{prop:anisotropic_trace}, Proposition~\ref{prop:conormal_product},
and Young's inequality, we get
\begin{equation}\label{eq:Linfty_graddivu_proof}
\begin{aligned}
\|\nabla\diver u(t)\|_{\mathcal H^{1,\infty}}^2
&\leq
C_4[1+P(Q(t))]
\left[
P\left(\Lambda_m^1(t)\right)
+
\|\nabla^2p(t)\|_{\mathcal H^{1,\infty}}^2
\right]  \\
&\leq
C_4[1+P(Q(t))]
C_\delta
\left\{
P\left(\Lambda_m^1(t)\right)
+
P\left(\|\Delta p(t)\|_{\mathcal H^1}^2\right)
\right\}  \\
&\quad
+
C_4\delta[1+P(Q(t))]
\|\Delta p(t)\|_{\mathcal H^2}^2 .
\end{aligned}
\end{equation}

This proves \eqref{eq:Linfty_grad_divu_H1infty}. Finally, \eqref{eq:Q_reduction} follows from
\eqref{eq:Linfty_lnrho_p_u}--\eqref{eq:Linfty_grad_divu_H1infty}. This completes the proof.
\end{proof}

\begin{lemma}\label{lem:Linfty_grad_u}
Let \(m\geq 6\). Then
\begin{equation}\label{eq:Linfty_grad_u_estimate}
\begin{aligned}
\|\nabla u(t)\|_{\mathcal H^{1,\infty}}^2
&\leq
C_{m+2}\Bigg(
\|(u_0,\nabla u_0)\|_{\mathcal H^{1,\infty}}^2
+
P(\Lambda_m^1(t))
+
P\left(\|\Delta p(t)\|_{\mathcal H^1}^2\right)  \\
&\quad
+
\varepsilon^2 t
\int_0^t
\|\nabla^2u(\tau)\|_{\mathcal H^4}^2\,d\tau
+
t\int_0^t
\left[1+P(\Lambda_m^1(\tau))+P(Q(\tau))\right]  \\
&\qquad
\times
\left[1+\varepsilon^2\|\Delta p(\tau)\|_{\mathcal H^2}^2\right]\,d\tau
\Bigg).
\end{aligned}
\end{equation}
\end{lemma}

\begin{proof}
The argument follows the \(L^\infty\)-estimate for the boundary vorticity
quantity in \cite{WangXinYong2015}, with additional electric terms treated by
the Poisson estimates.
Let \(\chi_0\) be a cut-off function supported away from the boundary, and let
\(\chi\) be a cut-off function supported in a boundary coordinate patch. For
the interior part, the standard Sobolev embedding gives, for \(|\alpha|\leq1\),
\begin{equation}\label{eq:grad_u_interior_Linfty}
\|\chi_0Z^\alpha\nabla u(t)\|_{L^\infty}^2
+
\|\chi_0\nabla u(t)\|_{L^\infty}^2
\leq
C\|u(t)\|_{\mathcal H^m}^2
\leq
C\Lambda_m^1(t).
\end{equation}
Thus it remains to estimate \(\chi Z^\alpha\nabla u\) and
\(\chi\nabla u\) near the boundary.

In a boundary coordinate patch, write
\[
x=\Psi^n(y,z)
=
\begin{pmatrix}
y\\
\psi(y)
\end{pmatrix}
-
z\,n(y),
\qquad y=(y_1,y_2),\quad z\geq0,
\]

where \(y=(y_1,y_2)\) is tangential and \(z\) is normal. For a function \(f\),
we denote its expression in these coordinates by
\[
\widetilde f(t,y,z):=f(t,\Psi^n(y,z)).
\]
Set
\[
(\bar\eta,\widetilde\rho,\widetilde u)(t,y,z)
:=
(\eta,\rho,u)(t,\Psi^n(y,z)).
\]

The local decomposition of the gradient gives
\begin{equation}\label{eq:grad_u_boundary_reduction}
\|\chi\nabla u(t)\|_{\mathcal H^{1,\infty}}^2
\leq
C_{m+1}
\left(
\|\bar\eta(t)\|_{\mathcal H^{1,\infty}}^2
+
\|\diver u(t)\|_{\mathcal H^{1,\infty}}^2
+
\|u(t)\|_{\mathcal H^{2,\infty}}^2
\right).
\end{equation}
Hence, by Lemma~\ref{lem:Linfty_basic},
\begin{equation}\label{eq:grad_u_boundary_by_eta}
\|\chi\nabla u(t)\|_{\mathcal H^{1,\infty}}^2
\leq
C_{m+2}
\left(
\|\bar\eta(t)\|_{\mathcal H^{1,\infty}}^2
+
P(\Lambda_m^1(t))
+
P\left(\|\Delta p(t)\|_{\mathcal H^1}^2\right)
\right).
\end{equation}
It remains to estimate \(\bar\eta\).

By \eqref{eq:eta_equation}, after transforming to the local coordinates, we
have
\begin{equation}\label{eq:bar_eta_original_equation}
\begin{aligned}
\widetilde\rho\,\partial_t\bar\eta
+
\widetilde\rho\,\widetilde u_y\cdot\nabla_y\bar\eta
+
\widetilde\rho\,(\widetilde u\cdot n)\partial_z\bar\eta
-
\mu\varepsilon\,\partial_{zz}\bar\eta
=
\widetilde F_\eta^0
+
\widetilde F_\eta^E
+
\mu\varepsilon
\left(
\widetilde{\Delta\eta}-\partial_{zz}\bar\eta
\right),
\end{aligned}
\end{equation}
with
\[
\bar\eta|_{z=0}=0.
\]
Since \(u\cdot n=0\) on \(\partial\Omega\), we have
\[
(\widetilde u\cdot n)(t,y,0)=0.
\]
Let
\[
a(t,y):=\partial_z(\widetilde u\cdot n)(t,y,0).
\]
Then
\[
\widetilde u\cdot n
=
za(t,y)
+
\left[
\widetilde u\cdot n-za(t,y)
\right].
\]
Therefore, \(\bar\eta\) satisfies the following normal parabolic equation:
\begin{equation}\label{eq:bar_eta_equation}
\begin{aligned}
\widetilde\rho\,\partial_t\bar\eta
+
\widetilde\rho\,\widetilde u_y\cdot\nabla_y\bar\eta
+
\widetilde\rho\,za(t,y)\partial_z\bar\eta
-
\mu\varepsilon\,\partial_{zz}\bar\eta
=
G,
\qquad
\bar\eta|_{z=0}=0,
\end{aligned}
\end{equation}
where
\begin{equation}\label{eq:G_definition}
\begin{aligned}
G
:={}&
\widetilde F_\eta^0
+
\widetilde F_\eta^E
+
\mu\varepsilon
\left(
\widetilde{\Delta\eta}-\partial_{zz}\bar\eta
\right)
-
\widetilde\rho
\left[
\widetilde u\cdot n-za(t,y)
\right]
\partial_z\bar\eta .
\end{aligned}
\end{equation}
Set \(\bar\eta_0=\bar\eta(0,y,z)\). Applying the \(L^\infty\)-estimate for the one-dimensional parabolic problem
\eqref{eq:bar_eta_equation}, and using Proposition~\ref{prop:conormal_product},
we get
\begin{equation}\label{eq:bar_eta_Linfty_framework}
\begin{aligned}
\|\bar\eta(t)\|_{\mathcal H^{1,\infty}}^2
\leq{}&
C\Bigg(
\|\bar\eta_0\|_{\mathcal H^{1,\infty}}^2
+
t\int_0^t
\|G(\tau)\|_{\mathcal H^{1,\infty}}^2\,d\tau  \\
&\quad
+
t\int_0^t
\big[1+P(\Lambda_m^1(\tau))+P(Q(\tau))\big]
\|\bar\eta(\tau)\|_{\mathcal H^{1,\infty}}^2\,d\tau
\Bigg).
\end{aligned}
\end{equation}

We now estimate \(G\). The terms generated by \(\widetilde F_\eta^0\), the
coordinate part of the Laplacian, and the coefficient commutators satisfy
\begin{equation}\label{eq:G_standard_estimate}
\begin{aligned}
&\|\widetilde F_\eta^0\|_{\mathcal H^{1,\infty}}^2
+
\varepsilon^2
\|\widetilde{\Delta\eta}-\partial_{zz}\bar\eta\|_{\mathcal H^{1,\infty}}^2  \\
&\leq
C_{m+2}
\big[
P(Q(t))
+
P(\Lambda_m^1(t))
+
\varepsilon^2\|\nabla^2u(t)\|_{\mathcal H^4}^2
\big].
\end{aligned}
\end{equation}

For the electric part, by \eqref{eq:eta_sources}, Proposition~\ref{prop:conormal_product},
and the elliptic estimate for the Poisson equation,
\begin{equation}\label{eq:G_electric_estimate}
\begin{aligned}
\|\widetilde F_\eta^E(t)\|_{\mathcal H^{1,\infty}}^2
&\leq
C_{m+2}
\left(
\|\nabla\rho\times\nabla\phi\|_{\mathcal H^{1,\infty}}^2
+
\|\rho\nabla\phi\|_{\mathcal H^{1,\infty}}^2
\right)  \\
&\leq
C_{m+2}[1+P(Q(t))]
P(\mathcal E_m(t)).
\end{aligned}
\end{equation}

It remains to control the last term in \eqref{eq:G_definition}. Since
\[
\widetilde u\cdot n-za(t,y)
=
\int_0^z
\int_0^s
\partial_{zz}(\widetilde u\cdot n)(t,y,r)\,dr\,ds,
\]
we have
\[
\left\|
\frac{\widetilde u\cdot n-za(t,y)}{z}
\right\|_{L^\infty}
\leq
C\|\partial_{zz}(\widetilde u\cdot n)\|_{L^\infty}.
\]
Moreover, since \(n\) is independent of \(z\) in the normal coordinates,
\begin{equation}\label{eq:uzz_normal_identity}
\partial_{zz}(\widetilde u\cdot n)
=
\partial_z\diver u
-
\partial_z(\Pi\partial_{y_1}u)_1
-
\partial_z(\Pi\partial_{y_2}u)_2 .
\end{equation}
Thus, by Lemma~\ref{lem:Linfty_basic} and the local decomposition of
\(\nabla u\),
\begin{equation}\label{eq:uzz_normal_estimate}
\|\partial_{zz}(\widetilde u\cdot n)(t)\|_{L^\infty}^2
\leq
C_3P(Q(t)).
\end{equation}
Consequently,
\begin{equation}\label{eq:G_normal_transport_estimate}
\begin{aligned}
&\left\|
\widetilde\rho
\left[
\widetilde u\cdot n-za(t,y)
\right]
\partial_z\bar\eta
\right\|_{\mathcal H^{1,\infty}}^2
\leq
C_4[1+P(Q(t))]
\left(
\|\bar\eta(t)\|_{\mathcal H^{1,\infty}}^2
+
P(\Lambda_m^1(t))
\right).
\end{aligned}
\end{equation}
Combining \eqref{eq:G_standard_estimate}--\eqref{eq:G_normal_transport_estimate},
we obtain
\begin{equation}\label{eq:G_total_estimate}
\begin{aligned}
\|G(t)\|_{\mathcal H^{1,\infty}}^2
\leq
C_4[1+P(Q(t))]
\Big[
&\|\bar\eta(t)\|_{\mathcal H^{1,\infty}}^2
+
P(\Lambda_m^1(t))
+
P(\mathcal E_m(t))  \\
&+
\varepsilon^2\|\nabla^2u(t)\|_{\mathcal H^4}^2
+
\varepsilon^2\|\Delta p(t)\|_{\mathcal H^2}^2
\Big].
\end{aligned}
\end{equation}
Substituting \eqref{eq:G_total_estimate} into
\eqref{eq:bar_eta_Linfty_framework} and applying Gronwall's inequality yields
\begin{equation}\label{eq:bar_eta_final_Linfty}
\begin{aligned}
&\|\bar\eta(t)\|_{\mathcal H^{1,\infty}}^2
\leq
C_{m+2}\Bigg(
\|\bar\eta_0\|_{\mathcal H^{1,\infty}}^2
+
\varepsilon^2 t
\int_0^t
\|\nabla^2u(\tau)\|_{\mathcal H^4}^2\,d\tau  \\
&+
t\int_0^t
\big[1+P(\Lambda_m^1(\tau))+P(Q(\tau))\big]
\big[1+\varepsilon^2\|\Delta p(\tau)\|_{\mathcal H^2}^2\big]\,d\tau
\Bigg).
\end{aligned}
\end{equation}

Finally, from \eqref{eq:eta_rewrite}, \eqref{eq:normal_by_eta},
\eqref{eq:grad_u_boundary_by_eta}, \eqref{eq:grad_u_interior_Linfty}, and
\eqref{eq:bar_eta_final_Linfty}, we obtain
\eqref{eq:Linfty_grad_u_estimate}. This completes the proof.
\end{proof}

\subsection{\texorpdfstring{Uniform estimate for \(\Delta p\)}
{Uniform estimate for Delta p}}
\begin{lemma}\label{lem:uniform_Delta_p}
Let \(m\geq6\) and \(0\leq t\leq\min\{T,1\}\). Then
\begin{equation}\label{eq:uniform_Delta_p_estimate}
\begin{aligned}
&\sup_{0\leq\tau\leq t}
\left(
\|\Delta p(\tau)\|_{\mathcal H^1}^2
+\varepsilon\|\Delta p(\tau)\|_{\mathcal H^2}^2
\right)
+\int_0^t\|\Delta p(\tau)\|_{\mathcal H^2}^2\,d\tau\\
&\leq C_{m+2}\Bigg[
P\left(\|(p_0,u_0,\nabla\phi_0)\|_{X_m^\varepsilon}^2\right)
+[1+P(Q(t))]\int_0^t
P\left(\|(p,u,\nabla\phi)(\tau)\|_{X_m^\varepsilon}^2\right)\,d\tau
\Bigg].
\end{aligned}
\end{equation}
\end{lemma}

\begin{proof}
Let
\[
q=\ln\rho,
\qquad
\dot u=u_t+u\cdot\nabla u.
\]
Taking divergence of \eqref{eq:momentum_curl_form} gives
\begin{equation}\label{eq:Delta_p_div_momentum}
-\varepsilon\Delta\diver u
+\frac{1}{2\mu+\lambda}\Delta p
=-\frac{1}{2\mu+\lambda}\diver(\rho\dot u)
+\frac{1}{2\mu+\lambda}\diver(\rho\nabla\phi).
\end{equation}
On the other hand, the continuity equation gives
\[
q_t+u\cdot\nabla q+\diver u=0.
\]
Hence
\[
-\Delta\diver u
=(\Delta q)_t+u\cdot\nabla\Delta q
+\Delta u\cdot\nabla q
+2\sum_{i=1}^3\partial_i u\cdot\nabla\partial_i q.
\]
Substituting this identity into \eqref{eq:Delta_p_div_momentum}, we obtain
\begin{equation}\label{eq:Delta_q_equation}
\begin{aligned}
&\varepsilon(\Delta q)_t
+\varepsilon u\cdot\nabla\Delta q
+\frac{1}{2\mu+\lambda}\Delta p\\
&=-\varepsilon\Delta u\cdot\nabla q
-2\varepsilon\sum_{i=1}^3\partial_i u\cdot\nabla\partial_i q
-\frac{1}{2\mu+\lambda}\diver(\rho\dot u)
+\frac{1}{2\mu+\lambda}\diver(\rho\nabla\phi).
\end{aligned}
\end{equation}
We follow the argument of \cite[Lemma~3.16]{WangXinYong2015}, with the
additional electric term estimated by the preceding Poisson estimates.
Applying \(Z^\alpha\), \(|\alpha|\leq2\), to
\eqref{eq:Delta_q_equation}, multiplying by \(Z^\alpha\Delta q\), and
integrating over \(\Omega\), we get
\begin{equation}\label{eq:Delta_q_energy_identity}
\begin{aligned}
&\frac\varepsilon2\frac{d}{dt}\|Z^\alpha\Delta q\|_{L^2}^2
+\frac{1}{2\mu+\lambda}
\int_\Omega Z^\alpha\Delta p\,Z^\alpha\Delta q\,dx\\
&=-\varepsilon\int_\Omega
Z^\alpha(u\cdot\nabla\Delta q)\,Z^\alpha\Delta q\,dx
-\varepsilon\int_\Omega
Z^\alpha(\Delta u\cdot\nabla q)\,Z^\alpha\Delta q\,dx\\
&\quad-2\varepsilon\sum_{i=1}^3\int_\Omega
Z^\alpha(\partial_i u\cdot\nabla\partial_i q)\,Z^\alpha\Delta q\,dx
-\frac{1}{2\mu+\lambda}\int_\Omega
Z^\alpha\diver(\rho\dot u)\,Z^\alpha\Delta q\,dx\\
&\quad+\frac{1}{2\mu+\lambda}\int_\Omega
Z^\alpha\diver(\rho\nabla\phi)\,Z^\alpha\Delta q\,dx.
\end{aligned}
\end{equation}

We first estimate the electric term. As in \eqref{eq:electric_low_final},
the Poisson equation yields
\[
\diver(\rho\nabla\phi)
=\nabla\rho\cdot\nabla\phi+\rho(\rho-1).
\]
By Proposition~\ref{prop:conormal_product} and
Proposition~\ref{prop:poisson_neumann_estimate},
\[
\int_0^t\|\diver(\rho\nabla\phi)\|_{\mathcal H^2}^2\,d\tau
\leq C_{m+2}[1+P(Q(t))]\int_0^tP(\mathcal E_m(\tau))\,d\tau.
\]
Therefore, Young's inequality gives
\begin{equation}\label{eq:Delta_p_electric_estimate}
\begin{aligned}
&\sum_{|\alpha|\leq2}\int_0^t
\left|\int_\Omega
Z^\alpha\diver(\rho\nabla\phi)\,Z^\alpha\Delta q\,dx\right|d\tau\\
&\leq\delta\int_0^t\|\Delta q\|_{\mathcal H^2}^2\,d\tau
+C_\delta C_{m+2}[1+P(Q(t))]
\int_0^tP(\mathcal E_m(\tau))\,d\tau.
\end{aligned}
\end{equation}

Next, since \(p=e^{\gamma q}\),
\[
\Delta p=\gamma p\,\Delta q+\gamma^2p|\nabla q|^2,
\qquad
Z^\alpha\Delta p=\gamma p\,Z^\alpha\Delta q
+\gamma[Z^\alpha,p]\Delta q
+\gamma^2Z^\alpha(p|\nabla q|^2).
\]
Using the positive lower bound of \(p\) and
Proposition~\ref{prop:conormal_product}, we obtain
\begin{equation}\label{eq:Delta_p_coercive}
\begin{aligned}
&\frac{1}{2\mu+\lambda}\sum_{|\alpha|\leq2}
\int_0^t\int_\Omega Z^\alpha\Delta p\,Z^\alpha\Delta q\,dx\,d\tau\\
&\geq c\int_0^t\|\Delta q\|_{\mathcal H^2}^2\,d\tau
-C_{m+2}[1+P(Q(t))]\int_0^t
\left(P(\mathcal E_m(\tau))+\|\Delta q\|_{\mathcal H^1}^2\right)\,d\tau.
\end{aligned}
\end{equation}

The remaining terms are estimated as in
\cite[(3.187)--(3.192)]{WangXinYong2015}.
We give the estimates here for completeness. By the momentum equation,
\eqref{eq:Linfty_grad_divu_Linfty}, and
Lemma~\ref{lem:normal_divu_estimate} at the lower orders, we have
\[
\begin{aligned}
\varepsilon^2\|\Delta u\|_{L^\infty}^2
&\leq C[1+P(Q(t))],\\
\varepsilon^2\int_0^t\|\Delta u\|_{\mathcal H^2}^2\,d\tau
&\leq C_{m+2}[1+P(Q(t))]\int_0^tP(\mathcal E_m(\tau))\,d\tau.
\end{aligned}
\]
The electric force \(\rho\nabla\phi\) in these estimates is controlled by
the same product and Poisson estimates used in
\eqref{eq:gradp_time_electric_estimate}.
Together with Proposition~\ref{prop:anisotropic_trace}, the local
decomposition of the Laplacian, and Young's inequality, this gives
\begin{equation}\label{eq:Delta_q_product_terms_estimate}
\begin{aligned}
&\varepsilon\sum_{|\alpha|\leq2}\int_0^t
\left|\int_\Omega
Z^\alpha(\Delta u\cdot\nabla q)\,Z^\alpha\Delta q\,dx\right|d\tau\\
&\quad+2\varepsilon\sum_{|\alpha|\leq2}\sum_{i=1}^3\int_0^t
\left|\int_\Omega
Z^\alpha(\partial_i u\cdot\nabla\partial_i q)\,Z^\alpha\Delta q\,dx\right|d\tau\\
&\leq\delta\int_0^t\|\Delta q\|_{\mathcal H^2}^2\,d\tau
+C_\delta\varepsilon^2\int_0^t\|\nabla^2u\|_{\mathcal H^3}^2\,d\tau\\
&\quad+C_\delta C_{m+2}[1+P(Q(t))]\int_0^t
\left(P(\mathcal E_m(\tau))
+\varepsilon^2\|\Delta q\|_{\mathcal H^2}^2
+P(\|\Delta q\|_{\mathcal H^1}^2)\right)\,d\tau.
\end{aligned}
\end{equation}
For the transport term, integration by parts, \(u\cdot n=0\), and the
Hardy inequality for \((u\cdot N)/\varphi(z)\) yield
\begin{equation}\label{eq:Delta_q_transport_estimate}
\begin{aligned}
&\varepsilon\sum_{|\alpha|\leq2}\int_0^t
\left|\int_\Omega
Z^\alpha(u\cdot\nabla\Delta q)\,Z^\alpha\Delta q\,dx\right|d\tau\\
&\leq\delta\int_0^t\|\Delta q\|_{\mathcal H^2}^2\,d\tau
+C_\delta\varepsilon^2\int_0^t\|\nabla^2u\|_{\mathcal H^3}^2\,d\tau\\
&\quad+C_\delta C_{m+2}[1+P(Q(t))]\int_0^t
\left(P(\mathcal E_m(\tau))
+\varepsilon^2\|\Delta q\|_{\mathcal H^2}^2\right)\,d\tau.
\end{aligned}
\end{equation}

For the material derivative term, we use
\[
\diver(\rho\dot u)
=\rho(\diver u)_t+\rho u\cdot\nabla\diver u
+\nabla\rho\cdot\dot u
+\rho\sum_{i,j=1}^3\partial_i u_j\partial_j u_i.
\]
Using \eqref{eq:pressure_div_relation}, the Hardy inequality, and
Proposition~\ref{prop:conormal_product}, we obtain
\[
\int_0^t\|\diver(\rho\dot u)\|_{\mathcal H^2}^2\,d\tau
\leq C_{m+2}[1+P(Q(t))]\int_0^tP(\mathcal E_m(\tau))\,d\tau.
\]
Consequently,
\begin{equation}\label{eq:Delta_q_material_estimate}
\begin{aligned}
&\sum_{|\alpha|\leq2}\int_0^t
\left|\int_\Omega
Z^\alpha\diver(\rho\dot u)\,Z^\alpha\Delta q\,dx\right|d\tau\\
&\leq\delta\int_0^t\|\Delta q\|_{\mathcal H^2}^2\,d\tau
+C_\delta C_{m+2}[1+P(Q(t))]
\int_0^tP(\mathcal E_m(\tau))\,d\tau.
\end{aligned}
\end{equation}

Combining \eqref{eq:Delta_q_energy_identity}--
\eqref{eq:Delta_q_material_estimate} and choosing \(\delta>0\)
sufficiently small, we obtain
\begin{equation}\label{eq:Delta_q_H2_est}
\begin{aligned}
&\varepsilon\sup_{0\leq\tau\leq t}\|\Delta q(\tau)\|_{\mathcal H^2}^2
+c\int_0^t\|\Delta q\|_{\mathcal H^2}^2\,d\tau\\
&\leq C\varepsilon\|\Delta q_0\|_{\mathcal H^2}^2
+C\varepsilon^2\int_0^t\|\nabla^2u\|_{\mathcal H^3}^2\,d\tau\\
&\quad+C_{m+2}[1+P(Q(t))]\int_0^t
\left(\varepsilon\|\Delta q\|_{\mathcal H^2}^2
+P(\|\Delta q\|_{\mathcal H^1}^2)+P(\mathcal E_m(\tau))\right)\,d\tau.
\end{aligned}
\end{equation}
On the other hand, since \(\partial_t\) commutes with the conormal vector
fields, the fundamental theorem of calculus gives
\begin{equation}\label{eq:Delta_q_H1_est}
\begin{aligned}
\sup_{0\leq\tau\leq t}\|\Delta q(\tau)\|_{\mathcal H^1}^2
&\leq\|\Delta q_0\|_{\mathcal H^1}^2
+2\int_0^t\|\Delta q\|_{\mathcal H^1}
\|\partial_t\Delta q\|_{\mathcal H^1}\,d\tau\\
&\leq\|\Delta q_0\|_{\mathcal H^1}^2
+\delta\int_0^t\|\Delta q\|_{\mathcal H^2}^2\,d\tau
+C_\delta\int_0^t\|\Delta q\|_{\mathcal H^1}^2\,d\tau.
\end{aligned}
\end{equation}
Taking \(\delta<c/2\), adding \eqref{eq:Delta_q_H1_est} to
\eqref{eq:Delta_q_H2_est}, and using \eqref{eq:L2_closing_estimate}, we get
\begin{equation}\label{eq:Delta_q_combined_estimate}
\begin{aligned}
&\sup_{0\leq\tau\leq t}
\left(\|\Delta q(\tau)\|_{\mathcal H^1}^2
+\varepsilon\|\Delta q(\tau)\|_{\mathcal H^2}^2\right)
+\int_0^t\|\Delta q\|_{\mathcal H^2}^2\,d\tau\\
&\leq C\left(\|\Delta q_0\|_{\mathcal H^1}^2
+\varepsilon\|\Delta q_0\|_{\mathcal H^2}^2+\mathcal E_m(0)\right)\\
&\quad+C_{m+2}[1+P(Q(t))]\int_0^t
\left(\varepsilon\|\Delta q\|_{\mathcal H^2}^2
+P(\|\Delta q\|_{\mathcal H^1}^2)+P(\mathcal E_m(\tau))\right)\,d\tau.
\end{aligned}
\end{equation}

It remains to return from \(q\) to \(p\). By \(p=e^{\gamma q}\),
Proposition~\ref{prop:conormal_product}, and the anisotropic Sobolev
inequality, we have, as in \cite[(3.148)]{WangXinYong2015},
\begin{equation}\label{eq:Delta_p_Delta_q_equivalence}
\begin{aligned}
\|\Delta q\|_{\mathcal H^1}^2
&\leq C\left[P\left(\|\Delta p\|_{\mathcal H^1}^2\right)
+P(\mathcal E_m)\right],\\
\|\Delta p\|_{\mathcal H^1}^2
&\leq C\left[P\left(\|\Delta q\|_{\mathcal H^1}^2\right)
+P(\mathcal E_m)\right],\\
\|\Delta q\|_{\mathcal H^2}^2
&\leq C\left[\|\Delta p\|_{\mathcal H^2}^2
+P\left(\|\Delta q\|_{\mathcal H^1}^2\right)+P(\mathcal E_m)\right],\\
\|\Delta p\|_{\mathcal H^2}^2
&\leq C\left[\|\Delta q\|_{\mathcal H^2}^2
+P\left(\|\Delta q\|_{\mathcal H^1}^2\right)+P(\mathcal E_m)\right].
\end{aligned}
\end{equation}
Thus the initial terms in \eqref{eq:Delta_q_combined_estimate} are bounded
by \(CP(\|(p_0,u_0,\nabla\phi_0)\|_{X_m^\varepsilon}^2)\), and its
time integral is bounded by
\[
C_{m+2}[1+P(Q(t))]\int_0^t
P\left(\|(p,u,\nabla\phi)(\tau)\|_{X_m^\varepsilon}^2\right)\,d\tau.
\]
Finally, applying \eqref{eq:Delta_p_Delta_q_equivalence} to the endpoint
norms and the time integral, using \eqref{eq:L2_closing_estimate}, and
enlarging \(P\) by H\"older's inequality in time, we obtain
\eqref{eq:uniform_Delta_p_estimate}. This completes the proof.
\end{proof}

\section{Proofs of Theorem~\ref{thm:main_uniform_regularity} and
Theorem~\ref{thm:main_inviscid_limit}}\label{sec:proofs}

With the uniform a priori estimates established in Section~\ref{sec:apriori},
the proofs of the local existence of solutions on a time interval
\([0,T_0]\) independent of \(\varepsilon\), the uniform regularity stated in
Theorem~\ref{thm:main_uniform_regularity}, and the inviscid convergence stated
in Theorem~\ref{thm:main_inviscid_limit} become standard, following the
compactness argument for the Navier--Stokes equations with Navier boundary
condition in \cite{MasmoudiRousset2012}. Indeed, one may first mollify the
initial data so that the classical local well-posedness theory for the viscous
problem applies. The uniform estimates obtained in Section~\ref{sec:apriori},
together with a continuous induction argument, yield a lifespan independent of
\(\varepsilon\) and the bound in
Theorem~\ref{thm:main_uniform_regularity}. Passing to the limit
\(\varepsilon\to0\) by compactness, and using the standard elliptic estimates
for the Poisson equation, gives a solution to the Euler--Poisson
system \eqref{eq:intro_EP} with boundary condition
\eqref{eq:intro_EP_boundary}. The Lipschitz bound contained in the uniform
estimates gives uniqueness of both the viscous solution and the limiting
Euler--Poisson solution, hence the whole family converges to the unique limit.
This proves Theorems~\ref{thm:main_uniform_regularity} and
\ref{thm:main_inviscid_limit}. We omit the details.

\vspace{2mm}

\section*{Acknowledgements}
This research is supported by the National Natural Science Foundation of China (No. 12471198, No.
12431018).

\section*{Data availability}  Data sharing is not applicable to this article as no data sets were generated or analysed during the current study.

\end{document}